\documentclass{article}

\usepackage[textwidth=15cm]{geometry}
\usepackage{fancyhdr}

\usepackage{xparse}

\usepackage{amssymb,amsfonts,amsmath}

\usepackage[mathscr]{eucal}

\usepackage{amsbsy}

\usepackage{stmaryrd}

\usepackage{bm}

\usepackage{braket}  

\usepackage{scalerel}

\usepackage{forloop}

\NewDocumentCommand{\Tr}{s}{\IfBooleanTF{#1}{\vphantom{\intercal}}{\intercal}}

\NewDocumentCommand{\InvHelper}{}{\scalebox{0.5}[1.0]{\( - \)}1}
\NewDocumentCommand{\Inv}{s}{\IfBooleanTF{#1}{\vphantom{\InvHelper}}{\InvHelper}}
\NewDocumentCommand{\PinHelper}{}{\dagger}
\NewDocumentCommand{\Pin}{s}{\IfBooleanTF{#1}{\vphantom{\PinHelper}}{\PinHelper}} 

\makeatletter
\def\@eathelper#1#2\end@eath#3{%
	\if\relax\detokenize{#2}\relax%
	#3%
	\else%
	#1{\@eathelper#2\end@eath{#3}}%
	\fi}
\newcommand\ea[2]{%
	\@eathelper#1\relax\end@eath{#2}}
\makeatother
 
\NewDocumentCommand{\Vc}{O{} m !g t' t"}{%
  \bm{#1{\mathbf{\MakeLowercase{#2}}}}%
  \IfValueT{#3}{_{#3}}%
  \IfBooleanTF{#4}{^{\Tr}}{%
  	\IfBooleanT{#5}{^{\Tr*}}}%
}
\RenewDocumentCommand{\Vc}{O{} m !g t' t"}{%
  \bm{\ea{#1}{\mathbf{\MakeLowercase{#2}}}}%
  \IfValueT{#3}{_{#3}}%
  \IfBooleanTF{#4}{^{\Tr}}{%
    \IfBooleanT{#5}{^{\Tr*}}}%
}

\NewDocumentCommand{\Mx}{O{} m !g t' t"}{
  \bm{#1{\mathbf{\MakeUppercase{#2}}}}%
  \IfValueT{#3}{_{#3}}%
  \IfBooleanTF{#4}{^{\Tr}}{%
    \IfBooleanT{#5}{^{\Tr*}}}%
}
\RenewDocumentCommand{\Mx}{O{} m !g t' t"}{
  \bm{\ea{#1}{\mathbf{\MakeUppercase{#2}}}}%
  \IfValueT{#3}{_{#3}}%
  \IfBooleanTF{#4}{^{\Tr}}{%
    \IfBooleanT{#5}{^{\Tr*}}}%
}

\NewDocumentCommand{\Tn}{O{} m !g}{%
  \boldsymbol{#1{\mathscr{\MakeUppercase{#2}}}}%
  \IfValueT{#3}{_{#3}}%
}
\RenewDocumentCommand{\Tn}{O{} m !g}{%
  \boldsymbol{\ea{#1}{\mathscr{\MakeUppercase{#2}}}}%
  \IfValueT{#3}{_{#3}}%
}

\NewDocumentCommand{\Tm}{O{} m !g t' t"}{
  \bm{#1{\mathbf{\MakeUppercase{#2}}}_{(#3)}}%
  \IfBooleanTF{#4}{^{\Tr}}{%
    \IfBooleanT{#5}{^{\Tr*}}}%
}
\RenewDocumentCommand{\Tm}{O{} m !g t' t"}{
  \bm{\ea{#1}{\mathbf{\MakeUppercase{#2}}}_{(#3)}}%
  \IfBooleanTF{#4}{^{\Tr}}{%
	\IfBooleanT{#5}{^{\Tr*}}}%
}

\NewDocumentCommand{\tttLRexp}{mmmmmmm}{%
  \IfBooleanTF{#1}{%
    \IfBooleanTF{#2}{%
      \IfBooleanTF{#3}{%
        \IfBooleanTF{#4}%
        {\Biggl#5 #7 \Biggr#6}
        {\biggl#5 #7 \biggr#6}}%
      {\Bigl#5 #7 \Bigr#6}}%
    {\bigl#5 #7 \bigr#6}}%
  {#5 #7 #6}%
}
      
\NewDocumentCommand{\prn}{ssssm}{\tttLRexp{#1}{#2}{#3}{#4}{(}{)}{#5}}

\NewDocumentCommand{\ang}{ssssm}{\tttLRexp{#1}{#2}{#3}{#4}{\langle}{\rangle}{#5}}
\NewDocumentCommand{\Ang}{m}{\left\langle #1 \right\rangle}

\NewDocumentCommand{\crly}{ssssm}{\tttLRexp{#1}{#2}{#3}{#4}{\{}{\}}{#5}}

\NewDocumentCommand{\sqr}{ssssm}{\tttLRexp{#1}{#2}{#3}{#4}{\lbrack}{\rbrack}{#5}}
\NewDocumentCommand{\Sqr}{m}{\left\lbrack #1 \right\rbrack}

\NewDocumentCommand{\ceil}{ssssm}{\tttLRexp{#1}{#2}{#3}{#4}{\lceil}{\rceil}{#5}}

\NewDocumentCommand{\floor}{ssssm}{\tttLRexp{#1}{#2}{#3}{#4}{\lfloor}{\rfloor}{#5}}

\NewDocumentCommand{\dsqr}{ssssm}{\tttLRexp{#1}{#2}{#3}{#4}{\llbracket}{\rrbracket}{#5}}

\NewDocumentCommand{\nrm}{ssssm}{\tttLRexp{#1}{#2}{#3}{#4}{\|}{\|}{#5}}
\NewDocumentCommand{\Nrm}{m}{\left\| #1 \right\|}

\NewDocumentCommand{\abs}{ssssm}{\tttLRexp{#1}{#2}{#3}{#4}{\vert}{\vert}{#5}}

\NewDocumentCommand{\ttm}{O{k}}{\mathop{\times_{\mspace{-2mu}#1}}}

\NewDocumentCommand{\ttv}{O{k}}{\mathbin{\bar{\smash{\times}\vphantom{\ast}}_{\mspace{-2mu}#1}}}

\DeclareMathOperator{\rank}{rank}

\NewDocumentCommand{\qtext}{m}{\quad\text{#1}\quad}

\NewDocumentCommand{\dsub}{m m t_ m}{#1_{#2_{#4}}}

\NewDocumentCommand{\tttMlist}{m m m m m m m}{%
  #1_{1} #2 %
  \IfBooleanF{#5}{#1_{2} #2} %
  \IfBooleanTF{#6}{ %
    #1_{3} %
    \IfBooleanT{#7}{#2 #1_{4} } %
  }{ %
    \IfBooleanTF{#3}{\cdots}{\dots} #2 #1_{#4} %
  }%
}

\NewDocumentCommand{\tttMSlist}{m m m m m}{%
	#1_{#2} #3 %
	\IfBooleanTF{#4}{\mydots}{\myldots} #3 #1_{#5}
}

\NewDocumentCommand{\tttRMlist}{m m m m m m m}{%
  \IfBooleanTF{#6}%
  {\IfBooleanT{#7}{#1_{4} #2} #1_{3} #2 #1_{2} #2 #1_{1}}%
  {#1_{#4} \IfBooleanF{#5}{#2 #1_{#4-1}} #2 \cdots  #2 #1_{1}}%
}

\NewDocumentCommand{\tttSlist}{m m m m m m}{%
  #1_{1} %
  #2 \IfBooleanTF{#3}{\cdots}{\dots}
  #2 #1_{#5-1} %
  \IfValueT{#6}{#2 #6} %
  #2 #1_{#5+1} %
  #2 \IfBooleanTF{#3}{\cdots}{\dots}
  #2 #1_{#4} %
}

\NewDocumentCommand{\tttRSlist}{m m m m m m}{%
  #1_{#4} %
  #2 \IfBooleanTF{#3}{\cdots}{\dots}
  #2 #1_{#5+1} %
  \IfValueT{#6}{#2 #6} %
  #2 #1_{#5-1} %
  #2 \IfBooleanTF{#3}{\cdots}{\dots}
  #2 #1_{1} %
}

\NewDocumentCommand{\miwc}{s s t! O{i} O{d}}{%
  \tttMlist{#4}{,}{\BooleanFalse}{#5}{#3}{#1}{#2}%
}

\NewDocumentCommand{\siwc}{O{k} O{i} O{d} g}{%
  \tttSlist{#2}{,}{\BooleanFalse}{#3}{#1}{#4}%
}

\NewDocumentCommand{\minc}{s s t! O{i} O{d}}{%
  \tttMlist{#4}{}{\BooleanTrue}{#5}{#3}{#1}{#2}%
}

\NewDocumentCommand{\inds}%
{> { \ReverseBoolean } s O{m} O{j} O{1}}{%
	\tttMSlist{#3}{#4}{\IfBooleanF{#1}{,}}{#1}{#2}%
}

\NewDocumentCommand{\sinc}{O{k} O{i} O{d} g}{%
  \tttSlist{#2}{}{\BooleanTrue}{#3}{#1}{#4}%
}

\makeatletter
\let\ex\expandafter
\NewDocumentCommand{\newcommandstring}{m O{0} m}{%
	\ex\newcommand\csname #1\endcsname[#2]{#3}}
\newcounter{@lettercounter}
\NewDocumentCommand{\newlettercommand}{m O{0} m}{%
	\forloop{@lettercounter}{1}{\value{@lettercounter}<26}{
		\edef\letter{\alph{@lettercounter}}
		\edef\Letter{\Alph{@lettercounter}}
		\begingroup
		\edef\@defhelper{\endgroup
			\unexpanded{\newcommandstring}{#1}[#2]{{%
				\unexpanded{\edef\letter}{\letter}
				\unexpanded{\edef\Letter}{\Letter}
				\unexpanded{#3}}}}\@defhelper}}
\newcommand{\NewDocCommandString}[3]{
	\ex\NewDocumentCommand\csname #1\endcsname{#2}{#3}}
\newcommand{\NewLetterCommand}[3]{%
	\forloop{@lettercounter}{1}{\value{@lettercounter}<26}{
		\edef\letter{\alph{@lettercounter}}
		\edef\Letter{\Alph{@lettercounter}}
		\begingroup
		\edef\@defhelper{\endgroup
			\unexpanded{\NewDocCommandString}{#1}{\unexpanded{#2}}{{%
				\unexpanded{\edef\letter}{\letter}
				\unexpanded{\edef\Letter}{\Letter}
				\unexpanded{#3}}}}\@defhelper}}
\makeatother

\usepackage{xcolor}

\catcode`\|=12\relax

\RequirePackage{tikz}
\usetikzlibrary{calc,trees,positioning,arrows,%
	chains,shapes.geometric,decorations.pathreplacing,%
	decorations.pathmorphing,shapes,matrix,shapes.symbols}

\tikzset{
	>=stealth',
	edgevals/.style={
		anchor=south,
		font=\footnotesize},
	valuenodes/.style={
		rectangle, 
		rounded corners, 
		draw=black,
		font=\footnotesize,
		minimum height=2em, 
		minimum width=2em,
		text centered, 
		on chain},
	line/.style={draw, thick, <-},
	element/.style={
		tape,
		top color=white,
		bottom color=blue!50!black!60!,
		minimum width=8em,
		draw=blue!40!black!90, very thick,
		text width=10em, 
		minimum height=3.5em, 
		text centered, 
		on chain},
}

\newcommand{\SPMscheme}{\begin{tikzpicture}
		[node distance=1cm,
		start chain,]
		\node[valuenodes,fill=blue!40] (T) {$\TT=\displaystyle {\sum_{i=1}^{r} \lambda_i \a_{i}^{\otimes m}}$};
		\node[valuenodes] (A) {$\cA = \displaystyle\operatorname{Span} \{\a_{1}^{\otimes n},\dots,\a_{r}^{\otimes n}\}$};
		\begin{scope}[start branch=hoejre]
			\node[valuenodes, on chain=going below, yshift=.4cm, xshift=.3cm] (deflate)
			{$\displaystyle \begin{array}{rcl}
					\TT&\gets& \TT-\lambda_i \a_{i}^{\otimes m}%
				\end{array}$};
		\end{scope}
		\node[valuenodes,fill=orange!50] (ain) {$\displaystyle a_i$};
		\node[valuenodes,fill=orange!50] (lai) {$\displaystyle \lambda_i$};
		
		\draw[->] (T) -- node[edgevals] {(A)} (A); %
		\draw[->] (A) -- node[edgevals] {(B)} (ain);%
		\draw[->] (ain) -- node[edgevals] {(C)} (lai); %
		\draw[->] (lai.south)  -|+(0,0)|- node[edgevals,xshift=-1.5cm] {(C)} (deflate.east);
		\draw[->] (deflate.west) |-+(0,0)-| (T.south);
\end{tikzpicture}}
\usepackage{pgfplots}
\usepackage{pgfplotstable}
\usetikzlibrary{pgfplots.colorbrewer}
\pgfplotsset{compat=1.17}

\usepackage[style=base]{caption}
\usepackage{subcaption}

\usepackage{comment}

\usepackage[T1]{fontenc}
\usepackage{lmodern}
\rmfamily
\DeclareFontShape{T1}{lmr}{b}{sc}{<->ssub*cmr/bx/sc}{}
\DeclareFontShape{T1}{lmr}{bx}{sc}{<->ssub*cmr/bx/sc}{}

\usepackage{enumitem}

\usepackage[square,colon,numbers]{natbib}

\usepackage{algorithm, algpseudocode}

\usepackage{hyperref}

\usepackage{amsmath, amsthm, cleveref}

\theoremstyle{plain}
\newtheorem{theorem}{Theorem}[section]
\newtheorem{proposition}[theorem]{Proposition}
\newtheorem{lemma}[theorem]{Lemma}

\theoremstyle{definition}
\newtheorem{definition}[theorem]{Definition}

\theoremstyle{remark}
\newtheorem{remark}[theorem]{Remark}

\crefname{section}{Section}{Sections}
\crefname{subsection}{Subsection}{Subsections}
\crefname{appendix}{Appendix}{Appendices}

\crefname{theorem}{Theorem}{Theorems}
\crefname{proposition}{Proposition}{Propositions}
\crefname{lemma}{Lemma}{Lemmas}
\crefname{conjecture}{Conjecture}{Conjectures}

\crefname{definition}{Definition}{Definitions}

\crefname{remark}{Remark}{Remarks}
\crefname{example}{Example}{Examples}
\crefname{fact}{Fact}{Facts}

\crefname{algorithm}{Algorithm}{Algorithms}
\crefname{figure}{Figure}{Figures}
\crefname{table}{Table}{Tables}

\makeatletter
\def\amsbb{\use@mathgroup \M@U \symAMSb}
\makeatother
\usepackage{bbold}

\let\mathbb\amsbb

\usepackage{fixme}
\fxsetup{
  status=draft,
  margin,
  theme=color,
  silent,
  marginface=\linespread{1}\footnotesize,
  author=
}
\definecolor{fxtarget}{rgb}{0.4,0.4,0.6}
\FXRegisterAuthor{jk}{jke}{{\textbf{JK}}}
\FXRegisterAuthor{jp}{jpe}{{\textbf{JP}}}

\makeatletter 
\@mparswitchfalse%
\makeatother
\normalmarginpar

\algsetblockdefx[Section]{Section}{EndSection}{}{0pt}[2]%
{\hspace{-50cm} \hfill \rule[2.2ex]{#2}{.1pt}\vspace{-2ex}\\%
	\hfill\textsc{#1} \vspace{-3ex}} 
[0]{\vspace{-2.3ex}}
\algrenewcommand\algorithmicrequire{\textbf{Input:}}
\algrenewcommand\algorithmicensure{\textbf{Output:}}
\algrenewtext{EndWhile}{}%
\algrenewtext{EndFor}{}%
\algrenewtext{EndIf}{}%

\newenvironment{keywords}{%
	\begin{center}
		\begin{minipage}{.884\textwidth}
			\textbf{Keywords:}}{
		\end{minipage}
	\end{center}}

\usepackage{etoolbox}          %

\DeclareMathOperator*{\Sym}{sym}
\let\sym\Sym
\DeclareMathOperator*{\Span}{span}

\newcommand{\mat}{\mathop{\mathtt{mat}}}
\newcommand{\opvec}{\mathop{\mathtt{vec}}}
\newcommand{\reshape}{\mathop{\mathtt{reshape}}}
\newcommand{\rgrad}{\operatorname{grad}}
\newcommand{\rhess}{\operatorname{Hess}}

\NewDocumentCommand{\sop}{O{m} m}{#2^{\otimes #1}}

\newcommand{\R}{{\mathbb{R}}}
\newcommand{\Z}{{\mathbb{Z}}}

\NewDocumentCommand{\W}{O{}}{\Tn[#1]{W}}
\NewDocumentCommand{\U}{O{}}{\Tn[#1]{U}}
\NewDocumentCommand{\V}{O{}}{\Tn[#1]{V}}
\NewDocumentCommand{\TT}{O{}}{\Tn[#1]{T}}

\NewDocumentCommand{\Tlc}{}{t}

\NewDocumentCommand{\Id}{}{\Mx{I}}

\NewLetterCommand{c\Letter}{}{{\mathcal{\Letter}}}

\newlettercommand{bb\Letter}{{\mathbb{\Letter}}}

\newlettercommand{bf\letter}{{\mathbf{\letter}}}
\newlettercommand{bf\Letter}{{\mathbf{\Letter}}}

\newlettercommand{sf\letter}{{\mathsf{\letter}}}
\newlettercommand{sf\Letter}{{\mathsf{\Letter}}}

\NewDocumentCommand{\dblfact}{ m !o !O{#2-1}}
{\IfNoValueTF{#2}%
	{\frac{(2#1)!}{#1!2^{#1}}}%
	{\frac{(#2)!}{(#1)!2^{#1}}}}

\NewDocumentCommand{\bigger}{sss}{%
\IfBooleanTF{#1}{%
	\IfBooleanTF{#2}{%
		\IfBooleanTF{#3}%
		{\Biggl}
		{\biggl}}%
	{\Bigl}}%
{\bigl}}

\NewDocumentCommand{\PsiFn}{O{d}mmmm}%
{\Psi^{(#1)}\prn*{{#2}, \!{#3} \,, {#4}, \!{#5}}}

\NewDocumentCommand{\fGMM}{O{d}}
{F^{(#1)}}
\NewDocumentCommand{\fdeb}{O{d}}
{\varGamma^{(#1)}}

\algrenewcommand\algorithmicrequire{\textbf{Input:}}     
\algrenewcommand\algorithmicensure{\textbf{Output:}}

\NewDocumentCommand{\tB}{e{^}}
{\IfValueTF{#1}{\Mx[\tilde]{B}\vphantom{\Mx{B}}^{#1}}{\Mx[\tilde]{B}}}

\NewDocumentCommand{\sqrenum}{ssss t! m m}
{\edef\helpersqrenum{\IfBooleanT{#1}{*}\IfBooleanT{#2}{*}\IfBooleanT{#3}{*}\IfBooleanT{#4}{*}}
\expandafter\sqr\helpersqrenum{\tttMlist{#6}{,}{\BooleanFalse}{#7}{#5}{\BooleanFalse}{\BooleanFalse}}}
	
\NewDocumentCommand{\Sqrenum}{t! m m}
{\Sqr{\tttMlist{#2}{,\,}{\BooleanFalse}{#3}{#1}{\BooleanFalse}{\BooleanFalse}}}

\newcommand{\dotp}[2]{#1\cdot #2}
\newcommand{\dotpb}[2]{\left<#1 , #2\right>}

\newcommand{\colspan}{\operatorname{colspan}}

\RenewDocumentCommand{\a}{O{} t' t"}{%
	\bm{\ea{#1}{\mathbf{a}}}%
	\IfBooleanTF{#2}{^{\Tr}}{%
		\IfBooleanT{#3}{^{\Tr*}}}%
}
\NewDocumentCommand{\x}{O{} t' t"}{%
	\bm{\ea{#1}{\mathbf{x}}}%
	\IfBooleanTF{#2}{^{\Tr}}{%
		\IfBooleanT{#3}{^{\Tr*}}}%
}
\NewDocumentCommand{\y}{O{} t' t"}{%
	\bm{\ea{#1}{\mathbf{y}}}%
	\IfBooleanTF{#2}{^{\Tr}}{%
		\IfBooleanT{#3}{^{\Tr*}}}%
}
\RenewDocumentCommand{\v}{O{} t' t"}{%
	\bm{\ea{#1}{\mathbf{v}}}%
	\IfBooleanTF{#2}{^{\Tr}}{%
		\IfBooleanT{#3}{^{\Tr*}}}%
}

\newcommand{\dimL}{d}

\newcommand{\mydots}{\> \! \! \cdots}
\newcommand{\myldots}{\ldots}%

\makeatletter
\newcommand\mytag[2]{%
	\def\@currentlabel{#1}%
	(#1)\label{#2}%
}
\makeatother

\newcommand{\tageq}[1]{%
	\refstepcounter{equation}%
	\mytag{\theequation}{#1}}

\makeatletter
\newcommand\scsubsection{\@startsection{subsection}{2}{.25in}%
	{1.3ex\@plus .5ex \@minus .2ex}%
	{-.5em \@plus -.1em}%
	{\reset@font\normalsize\scshape\bfseries}}
\makeatother

\title{Subspace Power Method for Symmetric Tensor Decomposition%
\thanks{The authors contributed equally.}}

\date{}

\author{Joe Kileel\thanks{Department of Mathematics and Oden Institute, University of Texas at Austin, Austin, TX, USA
		(\texttt{jkileel@math.utexas.edu}).} %
	\and Jo\~ao M. Pereira\thanks{Department of Mathematics, University of Georgia, Athens, GA, USA
		(\texttt{jpereira@uga.edu}).}}

\begin{document}
	
\pagestyle{fancy}
\fancyhead{}
\fancyhead[R]{Subspace Power Method}
\fancyhead[L]{J. Kileel and J. M. Pereira}

\maketitle

\begin{abstract}
	We introduce the \textit{Subspace Power Method (SPM)} for calculating the CP decomposition of low-rank real symmetric tensors.
	This algorithm calculates one new CP component at a time, alternating between applying the shifted symmetric higher-order power method (SS-HOPM) to a certain modified tensor, constructed from 
	a matrix flattening of the original tensor; and using appropriate deflation steps.
	We obtain rigorous guarantees for SPM regarding convergence and global optima for input tensors of dimension $d$ and order $m$ of CP rank up to $\mathcal{O}(d^{\lfloor m/2\rfloor})$, via results in classical algebraic geometry and optimization theory.
	As a by-product of our analysis we prove that SS-HOPM converges unconditionally, settling a conjecture  in [Kolda, T.G., Mayo, J.R.: Shifted power method for computing tensor eigenpairs. SIAM Journal on Matrix Analysis and Applications 32(4), 1095–1124 (2011)].
	We present numerical experiments which demonstrate that SPM is efficient and robust to noise, being up to one order of magnitude faster than state-of-the-art CP decomposition algorithms in certain experiments.
	Furthermore, prior knowledge of the CP rank is not required by SPM. %
\end{abstract}


% ---- Inserted File ----
\section{Introduction}\label{sec:intro}

A tensor is a multi-dimensional array \cite{kolda2009tensor}. 
A symmetric tensor is an array unchanged by permutation of indices.   
That is, $\TT$ of order $m$ is symmetric if for each index $(j_1, \ldots, j_m)$ and permutation $\sigma$ it holds $\TT_{j_1, \mydots, j_m} = \TT_{j_{\sigma(1)}, \mydots, j_{\sigma(m)}}.$   

Symmetric tensors arise naturally in many data processing applications.  They occur as higher order moments of a dataset, generalizing the mean and covariance of a random vector; as derivatives of multivariate real-valued functions, generalizing the gradient and Hessian of a function; and as adjacency tensors for hypergraphs, generalizing the adjacency matrix of graphs.
Being able to \textit{decompose} symmetric tensors is critical in domains including blind source separation \cite{cardoso1998blind,foobi2007}, independent component analysis \cite{hyvarinen2004independent,podosinnikova2019overcomplete}, antenna array processing \cite{dogan1995applications, chevalier2005virtual}, 
telecommunications \cite{van1996analytical,chevalier1999optimal}, pyschometrics \cite{carroll1970analysis}, chemometrics \cite{bro1997parafac}, %
magnetic resonance imaging \cite{basser2007spectral} and latent variable estimation for Gaussian mixture models \cite{ge2015learning,pereira2022tensor}, topic models and hidden Markov models \cite{anandkumar2014tensor}. %

The present paper presents a new algorithm for computing the celebrated \textit{real symmetric CP decomposition:}
\begin{equation}\label{eq:tensordecdef}
\TT = \sum_{i=1}^r \lambda_i \a_{i}^{\otimes m}.	
\end{equation} 
In \eqref{eq:tensordecdef}, the left-hand side is the given input, a real symmetric tensor $\TT$ of size $\dimL \times ... \times \dimL$ ($m$ times).
The task is to compute the right-hand side,
where $r$ is an integer, $\lambda_i\in \R$ are scalars, $\a_{i}\in \R^\dimL$ are unit-norm vectors, and $\a_{i}^{\otimes m}$ denotes the $m$-th tensor power (or outer product) of $\a_{i}$, that is,
\begin{displaymath}
(\a_{i}^{\otimes m})_{j_1 \cdots  j_{m}} = (\a_{i})_{j_1} \> \! \! \cdots (\a_{i})_{j_{m}}.
\end{displaymath}
For any tensor $\TT$, there exists an expression of type \eqref{eq:tensordecdef}, since $\{ \a^{\otimes m} :  \Vert \a \Vert =1 \}$ spans the vector space of all symmetric tensors \cite{landsberg2012tensors}.
Thus one requires $r$ to be minimal in \eqref{eq:tensordecdef}, and declares this integer to be the \textit{CP rank} of $\TT$.
It is a crucial fact that, generically, CP decompositions are \textit{unique} for rank-deficient tensors:
if $\TT$ satisfies \eqref{eq:tensordecdef} with $r< \frac{1}{\dimL} \binom{\dimL+m-1}{m} = \mathcal{O}(\dimL^{m-1})$, $\dimL > 6$ and $(\lambda_i, \a_{i})$ are Zariski-generic, then the rank of $\TT$ is indeed $r$ and the minimal decomposition \eqref{eq:tensordecdef} is unique (up to permutation and sign flips of $\a_{i}$).  Generic uniqueness is a result in algebraic geometry 
\cite{chiantini2017generic,ballico2005weak,chiantini2002weakly}.

While low-rank symmetric CP rank decompositions exist and are generically unique, computing them is another matter.  Hillar--Lim showed tensor decomposition is NP-hard in general \cite{hillar2013most}. 
Nonetheless, a number of works have sought efficient algorithms for sufficiently low-rank tensors, e.g., \cite{harshman1970foundations,foobi2007,anandkumar2014tensor,kolda2015numerical,ge2017optimization,ma2016polynomial,hopkins2016fast,nie2017low, johnston2023computing}.  Conjecturally in theoretical computer science, there exist efficient algorithms  that succeed with high probability in decomposing random tensors of rank on the order of the square root of the number of tensor entries, but not so for ranks substantially more \cite{wein2023average}. %
From a numerical linear algebra standpoint, producing practically efficient methods -- even with restricted rank -- is a further challenge.

\subsection{Our contributions}

In this paper, we develop a numerical algorithm that accepts $\TT$ as input, and aims to output the minimal decomposition \eqref{eq:tensordecdef}, up to trivial ambiguities, provided 
\begin{equation}\label{eq:rank_threshold}
\begin{cases}
r \leq\binom{\dimL+n-2}{n-1} &\text{if  } m=2n-1,\medskip\\
r \leq\binom{\dimL+n-1}{n} - \dimL &\text{if  } m=2n.
\end{cases}
\end{equation}
In devising the method, we assume an exactly low-rank decomposition exists, though simple adjustments allow the method to run on noisy tensors.
Notably, the new algorithm outperforms existing state-of-art methods by roughly one order of magnitude in terms of speed, in experiments with ground-truth tensors of moderate rank.  The algorithm also does not require knowledge of $r$ in advance, and instead estimates the rank.	 Thirdly, the method is robust to additive noise. %
 
We call the method the \textit{Subspace Power Method (SPM)}. To give a glimpse, it consists of three parts.
\begin{enumerate}[label=(\Alph*), ref=item:SPM_step_\Alph*]
\item \textsc{Extract Subspace}: We flatten $\TT$ to a matrix, and compute its singular vector decomposition. 
From this, we extract the subspace of order-$n$ tensors spanned by $\a_{i}^{\otimes n}$, $i=1,\dots,r$, where $n = \lceil \frac{m}{2}\rceil$, denoted by $\cA$.
\item \textsc{Power Method}: We seek one rank-1 point $\a_{i}^{\otimes n}$ in the subspace $\mathcal{A}$.
For this end, SS-HOPM \cite{kolda2009tensor} for computing tensor eigenvectors is applied to an appropriately modified tensor, constructed using $\mathcal{A}$.
\item \textsc{Deflation}: We solve for the corresponding scalar $\lambda_i$, and update the low-rank matrix factorization to be that of the flattening of $\TT - \lambda_i a_i^{\otimes m}$.
\end{enumerate}
The pipeline repeats $r$ times, until all $(\lambda_i, \a_{i})$ are recovered.
\Cref{fig:SPM} shows a schematic of SPM. See \cref{alg:estd} in \cref{sec:tdalgdesc} for full details. 

\begin{figure} 
\begin{center}
\SPMscheme
\caption{\textit{Schematic of the Subspace Power Method.} The input is the symmetric tensor $\TT$ (the low-rank decomposition of $\TT$ is unknown). The output is $(\lambda_i,\a_{i})_{i=1}^r$. SPM has three steps: (A) \textsc{Extract Subspace}, (B) \textsc{Power Method}, (C) \textsc{Deflate}.} \label{fig:SPM}
\end{center}	
\end{figure} 

 The paper's other contribution is that we prove various theoretical guarantees. 
 These come in two flavors.
 Firstly, we characterize the \textit{global optima} in the reformulation of tensor decomposition used here.  Specifically in \cref{prop:tdspanV,prop:subspacerank1generic}, we use the trisecant lemma from algebraic geometry to show that the only rank-$1$ points in the subspace $\mathcal{A}$ (from Step A above) are the CP components tensored-up, i.e., $\a_{i}^{\otimes n}$ (up to scale).  
Secondly, we establish \textit{convergence guarantees} for the \textsc{Power Method} iteration (Step B above).  This analysis is summarized by \cref{thm:power}.  In particular, using the \L ojasiewicz inequality, we prove \textsc{Power Method} converges from any initialization. %
 In fact this is an important technical contribution: while proving it, we also settle a conjecture of Kolda--Mayo that their SS-HOPM method for computing Z-eigenvectors of symmetric tensors always converges (see page 1107 of \cite{kolda2011shifted}). 
 Previous convergence results applied to generic tensors \cite{kolda2011shifted} or to certain applications \cite{tang2021convergence}.
Qualitative bounds are also obtained on the rate of convergence for Step B.  
Further, we prove that \textsc{Power Method} converges to a second-order optimal point for almost all initializations, and each CP component vector $\pm \a_{i}$ is an attractive fixed point.  
 
\subsection{{Comparison to prior art}}
SPM (\cref{alg:estd}) integrates various ideas %
in the literature into a single natural algorithm for symmetric tensor decomposition, with innovations in computational efficiency and convergence theory.

Step A of SPM is a variation of the classical Catalecticant Method \cite{iarrobino1999power}.  This goes back to Sylvester of the $19$th century \cite{sylvester1851lx}. 
Although Sylvester's work is quite well-known, SPM is, to our knowledge, the first \textup{efficient numerical} method for tensor decomposition based on this formulation of tensor decomposition.
The work \cite{brachat2010symmetric} proposed using symbolic techniques with Sylvester's Catalecticant Method to decompose symmetric tensors, but that algorithm appears slow already on modest-sized examples.
Other related works are \cite{bhaskara2019smoothed, johnston2023computing}, which generalize \cite{foobi2007} and enjoy strong theoretical guarantees, but these lack an implementation or any numerical demonstrations at present. 
Further algebraic works are \cite{iarrobino1999power} and  \cite{oeding2013eigenvectors}.
Perhaps most related to ours is the latter by Oeding and Ottaviani.  
Importantly however, we differ in that \cite{oeding2013eigenvectors} proposes  using standard polynomial-solving techniques, e.g., Gr\"obner bases, to compute $(\lambda_i, \a_{i})$ after reformulating tensor decomposition via Sylvester's Catalecticant Method. Unfortunately, a Gr\"obner basis approach is impractical already for small-sized problem dimensions, since Gr\"obner bases have doubly exponential running time, and require exact arithmetic.  By contrast, %
SPM consists of fast numerical iterations 
and optimized numerical linear algebra. %

Next, Step B of SPM connects with Kolda-Mayo's SS-HOPM method for computing tensor eigenvectors \cite{kolda2011shifted}.  
Step B may be viewed as the  higher-order power method applied to a different symmetric tensor $\TT[\widetilde]$, constructed from the subspace extracted in Step A, see \eqref{eq:Tmodify}. %
SPM identifies the Z-eigenvectors of $\TT[\widetilde]$ (computable by SS-HOPM) with the CP tensor components $\a_{i}$ of $\TT$.  
This is important because the power method applied directly to $\TT$ does \textit{not} recover the CP components of $\TT$. 
In analyzing the convergence of Step~B we settle a conjecture of \cite{kolda2011shifted} on SS-HOPM in general.

Finally, Step C of SPM on deflation is related to Wedderburn's rank reduction formula \cite{chu1995rank}.  For maximal efficiency, we derive an optimized implementation of the procedure using Householder reflections, to avoid recomputation as much as possible.

In other respects, 
\cref{alg:estd} bears resemblance to De Lathauwer {et al.}'s Fourth-Order-Only Blind Identification (FOOBI) algorithm for fourth-order symmetric tensor decomposition \cite{foobi2007}. %
It is also related to the methods for finding low-rank elements in subspaces of matrices in \cite{podosinnikova2019overcomplete, fornasier2018identification, nakatsukasa2017finding}, and the asymmetric decomposition method in \cite{phan2015tensor} repeatedly reducing the tensor length in given directions. 

For third-order tensors, the simultaneous diagonalization algorithm \cite{leurgans1993decomposition} (often attributed to Jennrich) is an efficient and provable decomposition method for tensors of size $\dimL \times \dimL \times \dimL$ and rank less than or equal to $\dimL$.  It uses matrix eigendecompositions in a straightforward manner. For higher-order tensors of order $m \geq 4$, one may flatten into a third-order tensor and apply simultaneous diagonalization provided the rank is $\mathcal{O}(\dimL^{ \lfloor (m-1)/2 \rfloor})$. Nevertheless, the algorithm is known to suffer from numerical instabilities \cite{beltran2019pencil}, motivating the study of  methods that are more stable to noise. 

As far as leading practical algorithms go, many existing ones use direct nonconvex optimization, attempting to minimize the squared residual
\begin{equation}\label{eq:sq_res}
\Big{\|}\TT - \sum_{i=1}^{r} \lambda_i \a_{i}^{\otimes m}\Big{\|}^2.
\end{equation}
This is performed, e.g., by gradient descent or symmetric variants of alternating least squares \cite{kolda2015numerical}.
In \cref{sec:tdsimulations}, we compare SPM to state-of-the-art numerical methods, including a non-linear least squares method (NLS) \cite{vervliet2016}, that aims to minimize \eqref{eq:sq_res} using Gauss-Newton iterations. The comparison is done using standard random ensembles for symmetric tensors, as well as ``worse-conditioned'' tensors with correlated components.  
We also compare to quite different but state-of-the-art theoretical methods: simultaneous diagonalization \cite{leurgans1993decomposition}, the method of generating polynomials \cite{nie2017generating, nie2017low} and FOOBI \cite{foobi2007, cardoso1991super}.%

\subsection{Organization of the paper}

The paper is organized as follows. \Cref{sec:model} establishes notation and basic definitions.  %
\Cref{sec:tdalgdesc} details the SPM algorithm. %
\Cref{sec:tdpowermethod} analyzes the convergence of the power iteration (Step B in \cref{fig:SPM}).
\Cref{sec:tdsimulations} presents numerics, with comparisons of runtime and noise sensitivity to other methods.
\Cref{sec:discussion} concludes.
The appendices contain certain technical proofs.
Matlab and Python code for SPM is available at  
{\texttt{\url{https://www.github.com/joaompereira/SPM}}}.

% --- End Inserted File ---

% ---- Inserted File ----
\section{Definitions and Notation} \label{sec:model} %

\subsection{Tensors and tensor products}
Let $\cT_\dimL^m = (\R^{\dimL})^{\otimes m} \cong \R^{\dimL^m} $ denote the vector space of real tensors of \textit{order} $m$ and \textit{length} $\dimL$ in each mode. 
It is a Euclidean space, with the Frobenius inner product and norm.  
If $\TT \in \cT_\dimL^m$, then $\TT_{\inds} = \Tlc_{\inds}$ 
is the entry indexed by $(\inds*) \in [\dimL]^{m}$ where $[\dimL] = \{1, \ldots, \dimL\}$. 
For tensors $\TT \in  \cT_\dimL^{m}$ and $\Tn{U} \in \cT_\dimL^{\tilde m}$, their \textit{tensor product}  in $\cT_\dimL^{m+\tilde m}$ is 
\begin{equation}
(\TT\otimes \Tn{U})_{\inds*[m+\tilde m][j][1]}=\Tlc_{\inds[m]} u_{\inds[m+\tilde m][j][m+1]} \!\! \quad \forall \, (\inds*[m+\tilde m])\in [\dimL]^{m+\tilde m}.
\end{equation}
\noindent The \textit{tensor power} $\TT^{\otimes p}\in \cT_\dimL^{pm}$ is the tensor product of $\TT$ with itself $d$ times. 
The \textit{tensor dot product} (or \textup{tensor contraction}) between $\TT \in \cT_\dimL^m$ and $\Tn{U} \in \cT_\dimL^{\tilde m}$, with $m\ge \tilde m$, is the tensor in $\cT_\dimL^{m-\tilde m}$ defined by
\begin{equation}
(\TT\cdot \Tn{U})_{\inds*[m][j][\tilde m+1]}=\sum_{j_1=1}^\dimL\cdots\sum_{j_{\tilde m}=1}^\dimL \Tlc_{j_1 \mydots j_m}u_{j_1 \mydots  j_{\tilde m}}.
\end{equation}
If $m=\tilde m$, contraction coincides with the inner product, i.e., $\dotp{\TT}{\Tn{U}}=\dotpb{\TT}{\Tn{U}}$. %
For $\TT \in \cT_\dimL^m$, a (real normalized) \textit{Z-eigenvector/eigenvalue pair} 
$(\Vc{v}, \lambda) \in \R^{d} \times \R$ 
is a vector/scalar pair satisfying
$ \TT \cdot \Vc{v}^{\otimes (m-1)} = \lambda \Vc{v} $
and $\Vc{v} \in \mathbb{S}^{d-1}$, see \cite{lim2005singular, qi2005eigenvalues}.  Here $\mathbb{S}^{d-1}$ denotes the unit-sphere in $\mathbb{R}^d$ with respect to the Euclidean norm $ \| \cdot \| := \| \cdot \|_2$.

There is a useful relation between inner and outer products of tensors.
\begin{lemma}[Inner product of tensor products \cite{hackbusch2019tensor}]
	\label{lemma:inner_tensor_powers}
	For tensors $\TT_1,\TT_2 \in \cT_\dimL^m$, $\U_1,\U_2 \in \cT_\dimL^{\tilde m}$,
	we have 
	\begin{displaymath}
		\Ang{\TT_1 \otimes \U_1, \TT_2 \otimes \U_2}
		=  \Ang{\TT_1 , \TT_2}  \Ang{\U_1, \U_2}.
	\end{displaymath}
	In particular, for  vectors $\Vc{u},\Vc{v}\in \R^\dimL$,
	we have $\Ang{\Vc{v}^{\otimes m}, \Vc{u}^{\otimes m}} = \Ang{\Vc{v}, \Vc{u}}^{m}$.
\end{lemma}

\subsection{Symmetric tensors}

\begin{definition}\label{def:sym}
	A tensor $ \TT \in \cT_\dimL^m$ is \emph{symmetric} if it is unchanged by any permutation of indices, that is, %
	\begin{equation}
	\Tlc_{\inds} = \Tlc_{\inds[\sigma(m)][j][\sigma(1)]} \!\! \quad \forall \, (\inds*)\in [\dimL]^{m}  \textup{ and } \sigma\in \Pi^m,
	\end{equation}
	where $\Pi^m$ is the permutation group on $[m]$. We denote by $\cS_\dimL^m$ the vector space of real symmetric tensors of order $m$ and length $\dimL$. A tensor $\TT\in \cT^m_\dimL$ may be \textit{symmetrized} by %
	$\Sym : \cT^m_d \rightarrow \cS^m_d$  defined as
	\begin{equation}\label{eq:symmetrizing_operator}
	\Sym (\TT)_{\inds*} = \frac{1}{m!}\sum_{\sigma\in \Pi^m} \Tlc_{\inds[\sigma(m)][j][\sigma(1)]}\quad \forall \, (\inds*)\in [\dimL]^{m}.
	\end{equation}
\end{definition}

\begin{lemma}[\cite{hackbusch2019tensor}]
	\label{lemma:sym_orthogonal_projection}
	The $\sym$ operation  in \cref{def:sym} is an orthogonal projection
	and so is self-adjoint.
	In particular, for a vector $\Vc{v} \in \R^n$ and tensor $\TT \in \cT^{d}_n$,
	\begin{displaymath}
		\Ang{\sym(\TT),\sop{\Vc{v}}}=\Ang{\TT,\sop{\Vc{v}}}.
	\end{displaymath}
\end{lemma}

\subsection{Symmetric tensor decomposition}

\begin{definition}\label{def:decomp}
For a symmetric tensor $\TT \in \cS^m_\dimL$, a %
\textup{real symmetric CP decomposition} is an expression
\begin{equation}\label{eqn:tensorAgain}
\TT = \sum_{i=1}^{r} \lambda_i \a_{i}^{\otimes m}, 
\end{equation}
where $r \in \Z$ is smallest possible, $\lambda_i \in \R$, and $\a_{i} \in \R^\dimL$ (without loss of generality, $\Vert \a_{i} \Vert_{2} = 1$).  The minimal $r$ is the \textup{real symmetric CP rank} of $\TT$.
\end{definition}

\begin{remark}
Some guarantees in this paper hold only for \textit{Zariski-generic} $\a_{i}$ and $\lambda_i$ in \eqref{eqn:tensorAgain}.  This means that there exists a polynomial $p$ such that the guarantees are valid whenever \linebreak $p(\a_{1}, \ldots, \a_{r}, \lambda_1, \ldots, \lambda_r) \neq 0$ occurs, and furthermore $p(\a_{1}, \ldots, \a_{r}, \lambda_1, \ldots, \lambda_r) \neq 0$ holds for \textit{some} unit-norm $\a_{i} \in \R^{\dimL}$ and $\lambda_i \in \R$ (see \cite{harris1992algebraic} for background).
In particular, this implies the guarantees hold \text{with probability $1$} provided $\a_{i}$ and $\lambda_i$ are drawn from any absolutely 
	continuous probability distributions on the sphere and real line. 
\end{remark}

\subsection{Unfolding tensors}
\begin{definition}
We let $\zeta_{d_1,\dots,d_m}$ be the \textup{index map} that maps multi-indexes in $\prod[d_i]$ to the corresponding indices in $\left[\prod d_i\right]$ in lexicographic order. 
Formally, %
$$\zeta_{d_1,\dots,d_m}(i_1,\dots, i_m) = 1 + \sum_{j=1}^m (i_j - 1)\prod_{k=0}^{j-1} d_{k}$$
with the convention $d_0=1$. 
\end{definition}

\begin{definition}\label{def:flatten}
Let $\TT \in \cT_\dimL^{m}$. The \textup{vectorization} of $\TT$, denoted $\opvec(\TT) \in \R^{\dimL^m}$, is 
defined by 
$$
\opvec(\TT)_{\zeta_{d,\dots,d}(i_1,\dots, i_m)} = \TT_{i_1, \ldots, i_m},
$$
while \textup{unvectorization} of a vector in $\R^{\dimL^m}$ is defined as the inverse operation.

For integers $D_1$, $D_2$ such that $D_1 D_2 = d^m$, define the function $\reshape$
so that $\reshape(\TT) \in \mathbb{R}^{D_1 \times D_2}$ and
$$\reshape(\TT, D_1, D_2)_{i_1 i_2} = \opvec(\TT)_{\zeta_{D_1,D_2}(i_1, i_2)}.$$
For $1\le n< m$, we denote the \textup{$1$-to-$n$ matrix unfolding} (or $1$-to-$n$ matrix flattening) of $\TT$ by 
$\mat_n(\TT) = \reshape(\TT, d^n, d^{m-n})$. 
\end{definition}

The next definition is used to describe unfoldings of tensors with low CP rank.

\begin{definition} \label{def:khatri}
Let $\Mx{A} \in \R^{\dimL \times r}$ be a matrix with columns $\a_{1}, \ldots, \a_{r} \in \R^{d}$.  The \textup{$n$th Khatri-Rao power of $\Mx{A}$}, denoted $\Mx{A}^{\bullet n} \in \R^{\dimL^n \times r}$, is defined to be the matrix 
with columns $\opvec(\a_{1}^{\otimes n}), \ldots,\linebreak \opvec(\a_{r}^{\otimes n}) \in \R^{\dimL^n}$.
\end{definition}

% --- End Inserted File ---

% ---- Inserted File ----
\section{Algorithm Description}\label{sec:tdalgdesc}

In this section, we derive SPM (\cref{alg:estd}). %
The input is a real symmetric tensor $\TT \in \cS_\dimL^{m}$ ($m \geq 3$).  For purposes of method development, we assume that $\TT$ admits an exact low-rank decomposition \eqref{eq:tensordecdef},
where $r = \mathcal{O}\prn*{\dimL^{\lfloor\tfrac{m}{2}\rfloor}}$.
The output is $(\lambda_i, \a_{i})_{i=1}^{r}$ for $i = 1, \ldots, r$ (up to sign ambiguity).

For the remainder of this section, we fix a positive integer $n< m$, and consider only the $1$-to-$n$ matrix flattening $\mat_n$, abbreviated as $\mat$.
SPM applies to any matrix unfolding, but in our discussion and implementation of SPM we often set $n=\lceil \tfrac{m}{2}\rceil$ by default.
This choice maximizes the greatest rank for which \cref{alg:estd} 
works.

\subsection{Three steps}
It is convenient to divide the algorithm description into three steps: \textsc{Extract Subspace}, \textsc{Power Method} and \textsc{Deflate}.

\scsubsection*{Extract Subspace}

Observe the \textup{tensor decomposition} \eqref{eq:tensordecdef} of $\TT$ is equivalent to a \textup{matrix factorization}
of the flattening of $\TT$ (\cref{def:flatten}):
\begin{equation}\label{eq:mattensor}
\nonumber \mat(\TT) = \sum_{i=1}^r \lambda_i \mat\left(\a_{i}^{\otimes m}\right)= \sum_{i=1}^r \lambda_i \opvec(\a_{i}^{\otimes n}) \opvec(\a_{i}^{\otimes (m-n)})^{\Tr}.
\end{equation}
where $n = \lceil \tfrac{m}{2}\rceil$. 
Let $\Mx{A} \in \R^{\dimL \times r}$ be the matrix with columns $\a_{1}, \ldots, \a_{r} \in \R^\dimL$, and $\Mx{\Lambda} \in \R^{r \times r}$
be the diagonal matrix with entries $\lambda_1, \ldots, \lambda_r$. Then \eqref{eq:mattensor} reads 
\begin{equation} \label{eq:nicemat}
\mat(\TT) = \Mx{A}^{\bullet n} \Mx{\Lambda} \, (\Mx{A}^{\bullet (m-n)})^{\Tr}.
\end{equation}
Define the subspace 
\begin{equation} \label{eq:def_cA}
\cA \,= \, \operatorname{span}\{ \a_{1}^{\otimes n}, \ldots, \a_{r}^{\otimes n} \} \, \subset \, \cS_\dimL^n. 
\end{equation}
It is the column space of $\Mx{A}^{\bullet n} $ upon unvectorization.

\textsc{Extract Subspace} obtains an orthonormal basis for $\cA$ from $\mat(\TT)$ in \eqref{eq:nicemat}.

\begin{proposition}\label{prop:tdspanV}  The following statements hold true:
\begin{itemize}\setlength\itemsep{0.2em}
\item Let $p = \min(n, m-n)$ and
assume $\a_{1}^{\otimes p}, \ldots, \a_{r}^{\otimes p}$ are linearly independent, and $\lambda_1, \ldots, \lambda_r$ are nonzero. 
Then $\mat(\TT)$ has rank $r$.  Moreover 
if $\mat(\TT) = \Mx{U} \Mx{S} \Mx{V}^{\Tr}$ is a thin SVD, then the columns of $\Mx{U}$ give an orthonormal basis of $\cA$ (after unvectorization).
\item If $r \leq \binom{\dimL+p-1}{p}$ and $\a_{1}, \ldots, \a_{r}$ are Zariski-generic, then $\a_{1}^{\otimes p}, \ldots,\a_{r}^{\otimes p}$ are linearly independent.
\end{itemize}
\end{proposition}
\begin{proof}
	Assume $\a_{1}^{\otimes p}, \ldots, \a_{r}^{\otimes p} $ are linearly independent, or equivalently 
    $\rank(\Mx{A}^{\bullet p})=r$.  We claim this implies $\rank(\Mx{A}^{\bullet q})=r$ where $q=\max(n, m-n)$.
 Suppose $\alpha_1,\dots,\alpha_r \in \mathbb{R}$ satisfy $\alpha_1 \a_{1}^{\otimes q} + \cdots + \alpha_r \a_{r}^{\otimes q}=0$. Contract both sides $q-p$ times with a vector $\Vc{z} \in \mathbb{R}^d$ such that $\Vc{z}' \a_{i} \neq 0$ for all $i=1,\dots, r$
  to obtain $\alpha_1 (\Vc{z}' \a_{1})^{q-p} \a_{1}^{\otimes p} + \cdots + \alpha_r (\Vc{z}' \a_{r})^{q-p} \a_{r}^{\otimes p}=0$.  
  This implies $\alpha_i (\Vc{z}' \a_{i})^{q-p} = 0$ by linear independence of $\a_{1}^{\otimes p}, \ldots, \a_{r}^{\otimes p}$ for $i = 1, \ldots, r$, whence $\alpha_i = 0$ for $i=1, \ldots, r$ as claimed.

Now consider \eqref{eq:nicemat}.  Matrices $\Mx{A}^{\bullet n}$ and $\Mx{\Lambda} \, (\Mx{A}^{\bullet (m-n)})^{\Tr}$ have full column and row rank respectively, both equal to $r$.  Thus $\mat(\TT)$ has rank $r$, and its column space equals the column space of $\Mx{A}^{\bullet n}$.   Obviously, the column space can be computed by a thin SVD.
	
	For the second bullet, assume $r \leq \binom{\dimL+m-n-1}{m-n}$. 
Note $\a_{1}^{\otimes (m-n)}, \ldots, \a_{r}^{\otimes (m-n)} $ are linearly independent if and only all $r \times r$ minors of $\Mx{A}^{\bullet (m-n)}$  are nonzero. 
This is a Zariski-open condition on  $\a_{1}, \ldots, \a_{r}$.
It holds generically because it holds for some $\a_{1}, \ldots, \a_{r}$, as rank-1 symmetric tensors span $\cS_{\dimL}^{m-n}$ and $\dim(\cS_{\dimL}^{m-n})=\binom{\dimL+m-n-1}{m-n}$.
\end{proof}

In view of the proposition, SPM extracts the subspace $\cA \subset \cS_\dimL^n$ from a thin SVD of $\mat(\TT)$, see \cref{alg:estd}.
For noisy input tensors, we must truncate the full-rank SVD of $\mat(\TT)$.  A numerical threshold is chosen to set $r$, 
where all singular values below the threshold are discarded.
This procedure in the noisy case, and \cref{prop:tdspanV} in the noiseless case, enables SPM to estimate the rank without prior knowledge of $r$.

\scsubsection*{Power Method}

The next step of SPM is to find a rank-$1$ point in $\cA$.
The following result is the essential underpinning.

\begin{proposition}\label{prop:subspacerank1generic}
	Let $r \leq \binom{\dimL+n-1}{n} - \dimL$.  
 Then for Zariski-generic $\a_{1}, \ldots, \a_{r}$,  the only rank-$1$ tensors in  $\cA = \operatorname{span} \{\a_{1}^{\otimes n}, \ldots, \a_{r}^{\otimes n} \}\subset \cS_\dimL^n$ are $\a_{1}^{\otimes n}, \ldots, \a_{r}^{\otimes n}$ (up to scale).%
\end{proposition}%
\begin{proof}%
	This is a special case of the generalized trisecant lemma in algebraic geometry, see \cite[Prop.~2.6]{chiantini2002weakly} or \cite[Exer.~IV-3.10]{hartshorne2013algebraic}.
	The set of rank $\leq \nobreak 1$ tensors is an irreducible algebraic cone of dimension $\dimL$ linearly spanning its ambient space $\cS_\dimL^n$. 
	It is the affine cone over the Veronese variety, denoted by $[\mathcal{V}_{\dimL}^n]$. 
	Note that $\cA$ is a secant plane through $r$ general points on $[\mathcal{V}_{\dimL}^n]$.  The dimensions of $[\mathcal{V}_{\dimL}^n]$ and $\cA$ are subcomplimentary: 
	\begin{equation}
	\dim([\mathcal{V}_{\dimL}^n]) + \dim(\cA) = \dimL + r \leq \binom{\dimL+n-1}{n} = \dim(\cS_\dimL^n).
	\end{equation}
	Therefore the generalized trisecant lemma applies. 
 It implies that $[\mathcal{V}_{\dimL}^n]$ and $\cA$ have no unexpected intersection points.
 Precisely, %
	$[\mathcal{V}_{\dimL}^n] \cap \cA =  \operatorname{span}\{\a_{1}^{\otimes n} \} \cup \ldots \cup \operatorname{span}\{\a_{r}^{\otimes n} \}$.
\end{proof}

 In \textsc{Power Method}, we seek a rank-1 element in $\cA$ by solving the program:
\begin{equation}\label{eq:SPM-P}
\max_{\Vc{x} \in \mathbb{R}^d} \, F_{\cA}(\Vc{x}) \quad \text{subject to} \,\, \|\Vc{x}\| = 1, 	
\end{equation}
where 
\begin{align} \label{eq:def-FcA}
 F_{\cA}(\Vc{x}) = \|P_\cA(\Vc{x}^{\otimes n})\|^2, 
\end{align}
with  $P_\cA$ the orthogonal projector from  $\cT_\dimL^n$ onto  $\cA$.
The next result justifies \eqref{eq:SPM-P}.
\begin{proposition}\label{prop:program_global_maxima}
For all $\Vc{x}$ with $\| \Vc{x} \|=1$, we have $F_{\cA}(\Vc{x})\le 1$ with equality if and only if $\Vc{x}^{\otimes n}\in \cA$.  
If $r \leq \binom{\dimL+n-1}{n} - \dimL$ and $\cA = \operatorname{span} \{\a_{1}^{\otimes n}, \ldots, \a_{r}^{\otimes n} \}\subset \cS_\dimL^n$ where
$\a_{1}, \ldots, \a_{r}$ are Zariski-generic, then the global maxima of \eqref{eq:SPM-P} are precisely $\pm \a_{1}, \ldots, \pm \a_{r}$ with function value $1$.
\end{proposition}
\begin{proof}
	If $\|\Vc{x}\|=1$, then
	\begin{displaymath}
	\|P_{\cA}(\Vc{x}^{\otimes n})\|^2 \le \|P_{\cA}(\Vc{x}^{\otimes n})\|^2 + \|P_{\cA^{\perp}}(\Vc{x}^{\otimes n})\|^2 = \|\Vc{x}^{\otimes n}\|^2 = 1,
	\end{displaymath}
	where $P_{\cA^{\perp}}$ denotes orthogonal projection onto the orthogonal complement  $\cA^{\perp}$ of $\cA$, thus $F_{\cA}(\Vc{x})\le 1$.
	Moreover, $F_{\cA}(\Vc{x})=1$ if and only if $\|P_{\cA^{\perp}}(\Vc{x}^{\otimes n})\|^2 =0$ if and only if $\Vc{x}^{\otimes n}$ lies in $\cA$. 
 The second sentence is immediate from \cref{prop:subspacerank1generic}.
\end{proof}

 In \textsc{Power Method}, we use projected gradient descent to solve \eqref{eq:SPM-P}.  
We initialize $\Vc{x}$ as a random vector in the unit-sphere $\mathbb{S}^{d-1}$, and iterate 
\begin{equation}\label{eq:powermethodit}
		\Vc{x} \leftarrow \frac{\dotp{P_{\cA}(\Vc{x}^{\otimes n})}{\Vc{x}^{\otimes n-1}} + \gamma \Vc{x}} {\|\dotp{P_{\cA}(\Vc{x}^{\otimes n})}{\Vc{x}^{\otimes n-1}} + \gamma \Vc{x}\|}
\end{equation}
until convergence.  Here $\gamma > 0$ is a fixed constant, whose reciprocal is the step size; according to \cref{thm:meta-converge},  we may set $\gamma > \sqrt{\frac{n-1}{n}}$.
The iteration \eqref{eq:powermethodit} is calculated using $\Mx{U}$ obtained in \textsc{Extract Subspace} (see \cref{prop:tdspanV}).  Specifically if $\U_1, \dots, \U_r$ are the unvectorized columns of $\Mx{U}$, then
\begin{equation}\label{eq:PAdef2-first}
	P_\cA (\x^{\otimes n}) = \sum_{i = 1}^r  \dotpb{\U_i}{ \x^{\otimes n}} \U_i,
\end{equation}
and 
\begin{equation}\label{eq:PAdef3}
	P_\cA (\x^{\otimes n}) \cdot \x^{\otimes n-1}= \sum_{i = 1}^r  \dotpb{\U_i}{ \x^{\otimes n}} \dotp{\U_i}{\x^{\otimes n-1}}.
\end{equation}
Let $\Mx[\bar]{U} = \texttt{reshape}(\Mx{U}, \dimL^{n-1}, \dimL r)$ and $\Mx{W} = \texttt{reshape}(\opvec(\x^{\otimes n - 1})^{\Tr}\Mx[\bar]{U}, \dimL, r)$. If $\inds*[r][\Vc{w}]$ are the columns of $\Mx{W}$, it holds $\Vc{w}_i = \dotp{\U_i}{\x^{\otimes n-1}}$ and $\Vc{w}'_i \x = \dotpb{\U_i}{\x^{\otimes n}}$.  Thus \eqref{eq:PAdef3} is
\begin{displaymath}
P_\cA (\x^{\otimes n}) \cdot \x^{\otimes n-1} = \Mx{W} \Mx{W}' \x.
\end{displaymath}
Suppose now \eqref{eq:powermethodit} 
converges to $\Vc[\bar]{x} \in \R^\dimL$.  
We check if $F_\cA(\Vc[\bar]{x}) = 1$ (up to a tolerance).  
If so, we proceed with the \textsc{Deflate} step below, as $\Vc[\bar]{x}$ is a CP component under the conditions in the last sentence of \cref{prop:program_global_maxima}.
Otherwise, $\Vc[\bar]{x}$ is discarded and \textsc{Power Method} is repeated with a fresh random initialization. In \cref{sec:optimization_landscape}, it is observed that often \textsc{Power Method} converges to a CP component on its first try. 

\begin{remark}
The iteration \eqref{eq:powermethodit} is equivalent to the shifted symmetric higher-order power method of \cite{kolda2011shifted} applied to a certain modified tensor, different from $\TT$.  We explain this in \cref{eq:connection-sshopm}. 
That is why we call this step \textsc{Power Method}.
\end{remark}

\scsubsection*{Deflate}
The last step of SPM is \textsc{Deflate}.  
Given one CP component $\pm \a_i$ from \textsc{Power Method}, it calculates the corresponding coefficient $\pm \lambda_i$.  Then it removes the term $\lambda_i  \a_i^{\otimes m}$ from $\TT$ by appropriately updating the factorization of $\mat(\TT)$.

Assume without loss of generality we obtained $\a_{1}$ from \textsc{Power Method}.   
Define
\begin{align}
	\nonumber \Mx{W}_\tau &= \mat(\TT)- \tau \opvec(\a_{1}^{\otimes n}) \opvec(\a_{1}^{\otimes (m-n)})^{\Tr}
	\\&=(\lambda_1-\tau)\opvec(\a_{1}^{\otimes n}) \opvec(\a_{1}^{\otimes (m-n)})^{\Tr} +  \sum_{i=2}^r \lambda_i \opvec(\a_{i}^{\otimes n}) \opvec(\a_{i}^{\otimes (m-n)})^{\Tr}
	\label{eq:lambdaeq1}
\end{align}
for $\tau \in \mathbb{R}$.
If $\Mx{A}^{\bullet n}$ has full column rank, by \eqref{eq:lambdaeq1} and \cref{prop:tdspanV} it follows that  $\Mx{W}_{\tau}$ has rank $r-1$ if $\tau =\lambda_1$ and rank $r$ otherwise. 
This property determines $\lambda_1$.
A formula for $\lambda_1$ is given by Wedderburn rank reduction %
\cite[Theorem 1.1]{chu1995rank}:
\begin{equation}\label{eq:lambda_pseudoinverse}
\lambda_1 = \frac{1}{\opvec(\a_{1}^{\otimes (m-n)})^{\Tr} \mat(\TT)^{\dagger} \opvec(\a_{1}^{\otimes n})},	
\end{equation}
where $^\dagger$ denotes the Moore-Penrose pseudo-inverse. 
In our implementation, we use formulas for updating $\Mx{U}$ and $\Mx{V}$ directly without recalculating the thin SVD of the deflated flattened tensor:
\begin{displaymath}
(\Mx{U}, \Mx{S}, \Mx{V})\gets \operatorname{svd}(\Mx{W}_{\lambda_1}) = \operatorname{svd}\left(\sum_{i=2}^r \lambda_i \mat (\a_{i}^{\otimes m})\right).
\end{displaymath}
However for efficiency reasons, rather than storing and updating $\Mx{S}$, we store a matrix $\Mx{C}$, which we set initially to $\Mx{S}^{-1}$ and update throughout the algorithm enforcing that $\mat(\tilde \TT) = \Mx{U} \Mx{C}^{-1} \Mx{V}^{\Tr}$, where $\tilde \TT$ is the deflated tensor. In the following proposition we witness that it is more convenient to calculate $\lambda_1$ in terms of $\Mx{C}$, and that the update formulas for the factorization
are favorable.  %
\begin{proposition}\label{prop:deflate}
Let $\Mx{U}\in \R^{\dimL^{n} \times r}$ and $ \Mx{V}\in \R^{\dimL^{m-n} \times r}$ have orthonormal columns, and let $\Mx{C}\in \R^{r\times r}$ be nonsingular such that $\mat(\TT) = \Mx{U} \Mx{C}^{-1} \Mx{V}^{\Tr}$. (Here $\Mx{C}$ is not necessarily diagonal.) 
Suppose, after possibly relabeling and/or flipping sign, we obtain $\Vc{a}_1$ from \textsc{{Power Method}}. Then
\begin{itemize}
	\item We obtain the corresponding coefficient  as
	\begin{equation}
		\lambda_1 = \frac1{{\Vc{\beta}}^{\Tr} \Mx{C} {\Vc{\alpha}}}~~\text{where}~~{\Vc{\alpha}} = \Mx{U}^{\Tr} \opvec(\a_{1}^{\otimes n})~~\text{and}~~{\Vc{\beta}} = \Mx{V}^{\Tr} \opvec(\a_{1}^{\otimes m - n}).
	\end{equation}
	\item Let $\Mx{O}_{\Vc{\alpha}}$, $\Mx{O}_{\Vc{\beta}}$ be $r\times (r-1)$ matrices whose columns form orthonormal bases for $\Span\{\Mx{C} {\Vc{\alpha}}\}^\perp$ and $\Span\{\Mx{C}^{\Tr} {\Vc{\beta}}\}^\perp$, respectively. We update 
	\begin{displaymath}
	(\Mx{U}, \Mx{C}, \Mx{V}) \gets (\Mx[\tilde]{U}, \Mx[\tilde]{C}, \Mx[\tilde]{V}),~~~\text{where}~~\Mx[\tilde]{U} = \Mx{U} \Mx{O}_{\Vc{\beta}},~~  \Mx[\tilde]{V} = \Mx{V} \Mx{O}_{\Vc{\alpha}} ~~\text{and}~~\Mx[\tilde]{C} = \Mx{O}_{\Vc{\alpha}}^{\Tr} \Mx{C} \Mx{O}_{\Vc{\beta}}.
	\end{displaymath}
	The update guarantees that $\Mx[\tilde]{U}$, $ \Mx[\tilde]{V}$ have orthonormal columns, $ \Mx[\tilde]{C}$ is nonsingular and
	$\Mx[\tilde]{U}  \Mx[\tilde]{C}^{-1}  \Mx[\tilde]{V}^{\Tr} = \mat(\TT[\tilde])$, where $\TT[\tilde]$ is the deflated tensor
	\begin{equation}\label{eq:deflated-tensor}
	\TT[\tilde] = \TT - \lambda_1 \sop{\Vc{a}_1} = \sum_{i=2}^r \lambda_i \sop{\Vc{a}_i}.
	\end{equation}
	\item The columns of $\Mx[\tilde]{U}$ give an orthonormal basis for $\Span\{ \a_{2}^{\otimes n}, \ldots, \a_{r}^{\otimes n} \}$.%
\end{itemize}
\end{proposition}

\begin{proof}
The formula for $\lambda_1$ follows from \eqref{eq:lambda_pseudoinverse}:
\begin{equation*}
\lambda_1= \frac{1}{\opvec(\a_{1}^{\otimes (m-n)})^{\Tr} \mat(\TT)^{\dagger} \opvec(\a_{1}^{\otimes n})} = \frac{1}{\opvec(\a_{1}^{\otimes (m-n)})^{\Tr} \Mx{V} \Mx{C} \Mx{U}^{\Tr} \opvec(\a_{1}^{\otimes n})} = \frac{1}{{\Vc{\beta}}^{\Tr} \Mx{C} {\Vc{\alpha}}}
\end{equation*}

To show the remaining bullets, let $\Mx[\tilde]{S} = \Mx{C}^{-1} - \lambda_{1} {\Vc{\alpha}} {\Vc{\beta}}^{\Tr}$.  Then
\begin{align*}
\Mx{U}  \Mx[\tilde]{S} \Mx{V}^{\Tr} &= \Mx{U} \Mx{C}^{-1} \Mx{V}^{\Tr} - \lambda_{1} \Mx{U} \Mx{U}^{\Tr} \opvec(\a_{1}^{\otimes n}) (\Mx{V} \Mx{V}^{\Tr} \opvec(\a_{1}^{\otimes (m-n)}))^{\Tr} \\
&= \mat(\TT) - \lambda_{1} \opvec(\a_{1}^{\otimes n}) \opvec(\a_{1}^{\otimes (m-n)})^{\Tr} = \Mx{W}_{\lambda_{1}},
\end{align*}
where we used $\Mx{U} \Mx{U}^{\Tr} \opvec(\a_{1}^{\otimes n}) = \opvec(\a_{1}^{\otimes n})$, since $\opvec(\a_{1}^{\otimes n}) \in \cA = \texttt{colspan}(\Mx{U})$, and similarly $\Mx{V} \Mx{V}^{\Tr} \opvec(\a_{1}^{\otimes (m-n)}) = \opvec(\a_{1}^{\otimes (m-n)})$. Our proof strategy is to show the following hold: \medskip

\begin{tabular}{cc}
\tageq{eq:deflateeq1}\ $ \Mx[\tilde]{S} = \Mx{O}_{\Vc{\beta}} \Mx{O}_{\Vc{\beta}}^{\Tr} \Mx[\tilde]{S} \Mx{O}_{\Vc{\alpha}} \Mx{O}_{\Vc{\alpha}}^{\Tr}$,	\hspace{2cm} & \tageq{eq:deflateeq2}\ $ \Mx[\tilde]{C} = (\Mx{O}_{\Vc{\beta}}^{\Tr} \Mx[\tilde]{S} \Mx{O}_{\Vc{\alpha}})^{-1}$. %
\end{tabular}\medskip\newline
These will imply $ \Mx[\tilde]{U} \Mx[\tilde]{C}^{-1}  \Mx[\tilde]{V}^{\Tr} = \Mx{U} \Mx{O}_{\Vc{\beta}} \Mx{O}_{\Vc{\beta}}^{\Tr} \Mx[\tilde]{S} \Mx{O}_{\Vc{\alpha}} \Mx{O}_{\Vc{\alpha}}^{\Tr} \Mx{V}^{\Tr} = \Mx{U} \Mx[\tilde]{S} \Mx{V}^{\Tr} = \Mx{W}_{\lambda_{1}}$. Since $\Mx{U}$ and $\Mx{O}_{\Vc{\alpha}}$ have orthonormal columns, $ \Mx[\tilde]{U}=\Mx{U} \Mx{O}_{\Vc{\alpha}}$ also has orthonormal columns (likewise for $ \Mx[\tilde]{V}$), and so the proposition follows.

We now prove \eqref{eq:deflateeq1}. Let $\mathbf{y} = \tfrac{\Mx{C} \Vc{\alpha}}{\|\Mx{C} {\Vc{\alpha}}\|}$, then the definition of $\Mx{O}_{\Vc{\alpha}}$ implies that $\left[\Mx{O}_{\Vc{\alpha}}\; \mathbf{y}\right]$ is a $r\times r$ orthogonal matrix, which implies $\Mx{O}_{\Vc{\alpha}} \Mx{O}_{\Vc{\alpha}}^{\Tr} = \Id - \mathbf{y} \mathbf{y}^{\Tr}$.
Moreover, since $\lambda_{1} =1/({\Vc{\beta}}^{\Tr} \Mx{C} {\Vc{\alpha}})$, we have
\begin{displaymath}
\Mx[\tilde]{S} \mathbf{y} = \frac{1}{\|\Mx{C} {\Vc{\alpha}}\|} (\Mx{C}^{-1} \Mx{C} {\Vc{\alpha}} - \lambda_{1} {\Vc{\alpha}} {\Vc{\beta}}^{\Tr} \Mx{C} {\Vc{\alpha}}) = \frac{1}{\|\Mx{C} {\Vc{\alpha}}\|}({\Vc{\alpha}} - {\Vc{\alpha}}) = 0.
\end{displaymath}
Therefore, $ \Mx[\tilde]{S} \Mx{O}_{\Vc{\alpha}} \Mx{O}_{\Vc{\alpha}}^{\Tr} =  \Mx[\tilde]{S} -  \Mx[\tilde]{S} \mathbf{y} \mathbf{y}^{\Tr} =  \Mx[\tilde]{S}$.  The verification of $ \Mx[\tilde]{S} = \Mx{O}_{\Vc{\beta}} \Mx{O}_{\Vc{\beta}}^{\Tr}  \Mx[\tilde]{S} $ is analogous, and \eqref{eq:deflateeq1} follows. 
Regarding \eqref{eq:deflateeq2}, using again that $ \Mx[\tilde]{S} =  \Mx[\tilde]{S} \Mx{O}_{\Vc{\alpha}} \Mx{O}_{\Vc{\alpha}}^{\Tr}$, with ${\Vc{\beta}}^{\Tr} \Mx{C} \Mx{O}_{\Vc{\beta}} = 0$ which follows from the definition of $\Mx{O}_{\Vc{\beta}}$, we obtain
\begin{align*}
\nonumber (\Mx{O}_{\Vc{\beta}}^{\Tr} \Mx[\tilde]{S} \Mx{O}_{\Vc{\alpha}})  \Mx[\tilde]{C}  &=  \Mx{O}_{\Vc{\beta}}^{\Tr} \Mx[\tilde]{S} \Mx{O}_{\Vc{\alpha}} \Mx{O}_{\Vc{\alpha}}^{\Tr} \Mx{C} \Mx{O}_{\Vc{\beta}}= \Mx{O}_{\Vc{\beta}}^{\Tr} \Mx[\tilde]{S} \Mx{C} \Mx{O}_{\Vc{\beta}},\\
\nonumber &= \Mx{O}_{\Vc{\beta}}^{\Tr}\Mx{C}^{{-}1} \Mx{C} \Mx{O}_{\Vc{\beta}} - \lambda_{1} \Mx{O}_{\Vc{\beta}}^{\Tr} {\Vc{\alpha}} {\Vc{\beta}}^{\Tr} \Mx{C} \Mx{O}_{\Vc{\beta}} = \Mx{O}_{\Vc{\beta}}^{\Tr} \Mx{O}_{\Vc{\beta}} = \Id.
\end{align*}
\end{proof}

In light of the proposition, \textsc{Deflate} proceeds as follows. 
Set ${\Vc{\alpha}} \gets \Mx{U}' \opvec(\a_{i}^{\otimes n})$, ${\Vc{\beta}} \gets  \Mx{V}' \opvec(\a_{i}^{\otimes m - n})$ and  $\lambda_i = 1 / {\Vc{\beta}}^{\Tr} \Mx{C} {\Vc{\alpha}}$, where $\Mx{C} = \Mx{S}^{-1}$.  We then calculate $\Mx{O}_{\Vc{\alpha}}$ and $\Mx{O}_{\Vc{\beta}}$, and update $(\Mx{U},\Mx{C},\Mx{V}) \gets (\Mx{U} \Mx{O}_{\Vc{\beta}}, \Mx{O}'_{\Vc{\alpha}} \Mx{C} \Mx{O}"_{\Vc{\beta}}, \Mx{V} \Mx{O}"_{\Vc{\alpha}})$.

By design, the procedure enjoys two nice properties.
Firstly, the columns of $\Mx[\tilde]{U}$ give an orthonormal basis of $\tilde{\cA} = \Span \{\a_{2}^{\otimes n}, \ldots, \a_{r}^{\otimes n}\}$. Thus, we can use $\Mx[\tilde]{U}$ for the next run of \textsc{Power Method}.
Secondly, the orthogonal matrices $\Mx{O}_{\Vc{\alpha}}$ and $\Mx{O}_{\Vc{\beta}}$ can be constructed efficiently using Householder reflections.  
We explain the implementation for $\Mx{O}_{\Vc{\alpha}}$; $\Mx{O}_{\Vc{\beta}}$ is implemented analogously.
	Set $\y= \tfrac{\Mx{C} \Vc{\alpha}}{\|\Mx{C} \Vc{\alpha}\|}$ and define $\Vc{z}\in \R^r$ by $z_r =   \sqrt{1 + |y_r|}$ and $z_i = \operatorname{sign}(y_r) y_i / z_r$ for $i=1,\dots,r-1$.
	It is easily checked $\|\mathbf{z}\|^2 = 2$, therefore the matrix $\Mx{H} = \Id - \Vc{z} \Vc{z}'$ is a Householder reflection, and the last column $\Mx{H}$ is $\operatorname{sign}(y_r) \y$.
 We pick $\Mx{O}_{\Vc{\alpha}}$ to be the first $r-1$ columns of $\Mx{H}$, which form an orthonormal basis for $\Span\{\y\}^\perp = \Span\{\Mx{C} {\Vc{\alpha}}\}^\perp$.
 Using Householder reflections is more efficient than explicitly forming $\Mx{O}_{\Vc{\alpha}}$, because we can exploit that $\Mx{O}_{\Vc{\alpha}}$ is a rank-1 update of the identity matrix to calculate the matrix product $\Mx{V} \Mx{O}_{\Vc{\alpha}}$ in $\mathcal{O}(\dimL^{m- n} r)$ time. It gives a speed-up compared to calculating this product naively, 
 which takes $\mathcal{O}(\dimL^{m- n} r^2)$ time.

\subsection{Full algorithm}

\textsc{Power Method} and \textsc{Deflate} repeat as subroutines, such that each CP component $(\lambda_i, \a_{i})$ is removed one at a time, until all components have been found.  %
The full Subspace Power Method is detailed in \cref{alg:estd} below.

\begin{algorithm}[ht]
	\caption{Subspace Power Method (SPM)}
	\label{alg:estd}
	\begin{algorithmic}%
		\Require generic $\TT \in \cS_\dimL^{m}$ of rank $r$ satisfying \eqref{eq:rank_formula} \\
		\text{Hyperparameters:} $\kappa>0$, $\zeta> 0$, $1 \le n<m$, $\gamma> \sqrt{\frac{n-1}{n}}$
		\Ensure rank $r$ and tensor decomposition $\{(\lambda_i, \a_{i})\}_{i=1}^r$
		\Section{Extract Subspace}{\linewidth}
		\State $(\Mx{U},\Mx{S},\Mx{V}) \gets \operatorname{svd}(\mat(\TT)) $ 
		\State $\Mx{C} \gets \Mx{S}^{-1}$
		\State $r \leftarrow \rank(\mat(\TT))$
		\EndSection
		\For{$i=1$ \textbf{to} $r$}
		\State $\Mx[\bar]{U} \gets \texttt{reshape}(\Mx{U}, \dimL^{n-1}, \dimL r)$
		\Section{Power Method}{\linewidth-\algorithmicindent}
		\State $\Vc{x} \gets \operatorname{random}(\mathbb{S}^{\dimL-1})$
		\Repeat
		\State $\Vc[\tilde]{x}\gets \Vc{x}$ 
		\State $\Mx{W} \gets \texttt{reshape}(\opvec(\x^{\otimes n - 1})^{\Tr}\Mx[\bar]{U}, \dimL, r)$\smallskip
		\State $\Vc{x} \gets \displaystyle \frac{\Mx{W} \Mx{W}' \x + \gamma \Vc{x}}{\|\Mx{W} \Mx{W}' \x + \gamma \Vc{x}\|}$
		\Until{$\|\Vc{x} - \Vc[\tilde]{x}\| < \kappa$}
		\If{$F_\cA(\Vc{x})> \zeta$} $\a_{i} \gets \Vc{x}$
		\Else ~repeat \textsc{Power Method}\EndIf \vspace{-\baselineskip}\EndSection \Section{Deflate}{\linewidth-\algorithmicindent}
	    \State ${\Vc{\alpha}} \gets \Mx{U}^{\Tr} \opvec(\a_{i}^{\otimes n})$, ${\Vc{\beta}} \gets  \Mx{V}^{\Tr} \opvec(\a_{i}^{\otimes m - n})$
	    \State $\lambda_i \gets  \Nrm{\Vc{\beta}}\Nrm{\Vc{\alpha}}  / ({\Vc{\beta}}^{\Tr} \Mx{C} {\Vc{\alpha}})$
	    \State $\Mx{O}_{\Vc{\alpha}} \gets$ orthonormal basis of $\Span\{\Mx{C} {\Vc{\alpha}}\}^\perp$
	    \State $\Mx{O}_{\Vc{\beta}} \gets$ orthonormal basis of $\Span\{\Mx{C}^{\Tr} {\Vc{\beta}}\}^\perp$
	    \State $(\Mx{U},\Mx{C},\Mx{V}) \gets (\Mx{U} \Mx{O}_{\Vc{\beta}}, \Mx{O}_{\Vc{\alpha}}^{\Tr} \Mx{C} \Mx{O}_{\Vc{\beta}}, \Mx{V} \Mx{O}_{\Vc{\alpha}})$
		\EndSection
		\EndFor
		\Return $r$ and $\{(\lambda_i, \a_{i})\}_{i=1}^r$
	\end{algorithmic}
\end{algorithm}

\subsection{Practical considerations}
\label{sec:practical}

\subsubsection{Computational and storage costs}
\label{sec:comp_and_store_costs}
The computational costs of \cref{alg:estd} are as follows.  
\textsc{Extract Subspace} computes $\operatorname{svd}(\mat(\TT))$ upfront in $\mathcal{O}(\dimL^{m+k})$, where $k=\min(n, m-n)$. If an upper bound for the rank $\tilde r\ge r$ is known a priori, this drops to $\mathcal{O}(\dimL^{m} \tilde r)$ (e.g., using randomized linear algebra). %
Suppose $s \leq r$ components $(\lambda_i, \a_{i})$ are yet to be found.  
In \textsc{Power Method}, %
 each iteration of \eqref{eq:powermethodit} costs $\mathcal{O}(s \dimL^n)$, the price of applying $P_{\cA}$.  
In \textsc{Deflate}, computing $\Vc{\alpha}$ and $\Vc{\beta}$ cost $\mathcal{O}(s \dimL^n)$ and $\mathcal{O}(s \dimL^{m-n})$ respectively, computing $\lambda_i$ is $\mathcal{O}(s^2)$ and updating $(\Mx{U}, \Mx{C}, \Mx{V})$ is $\mathcal{O}(s (\dimL^{n} + \dimL^{m-n} + s))$.
The storage costs are $\mathcal{O}(s(\dimL^{m-n}+\dimL^{n} + s))$, corresponding to storing the matrix factorization of $\mat(\TT)$. The storage of this matrix factorization dominates the other storage costs.

\subsubsection{Maximal rank}\label{sec:rank}
If the CP components are Zariski-generic, SPM can work up to the ranks in \cref{prop:tdspanV,prop:subspacerank1generic}. 
These require $r\leq \binom{\dimL+p-1}{p}$ (where $p=\min(n, m-n))$, and $r\leq \binom{\dimL+n-1}{n}-\dimL$, respectively. Together, the conditions imply
\begin{equation}\label{eq:rank_formula}
r \, \leq \, \binom{\dimL+ p-1}{p} \, - \, \delta_{p=n} \, \dimL,
\end{equation}
where $\delta_{p=n} = 1$ if $p=n$ and $\delta_{p=n} = 0$ otherwise. The maximal rank is obtained when $n=\lceil m/2 \rceil$ and $p=\lfloor m/2 \rfloor$.
The formulae for the maximal tensor rank for the first few tensor orders are
\begin{equation}\label{eq:rank_formula_examples}
\begin{array}{c l m{20pt} c l}
m=3: & r\le d & & m=4: & r\le \frac12 d(d-1)\medskip\\
m=5: & r\le \frac12 d(d+1) & & m=6: & r\le \frac16 d(d^2+3d-4).
\end{array}
\end{equation}

\subsubsection{Eigendecomposition}
\label{rem:even}
	When $m$ is even and $n= m/2$,  $\mat(\TT)$ is a symmetric matrix. 
Then an eigendecomposition algorithm may be used in place of SVD in \textsc{Extract Subspace}. Furthermore, a symmetric variant of \cref{prop:deflate} holds where  $\Vc{\alpha}=\Vc{\beta}$, $\Mx{U} = \Mx{V}$ and $\Mx{C}$ is symmetric.

\subsubsection{Unique rows and columns}
Since $\TT$ is a symmetric tensor, $\mat(\TT)$ has many repeated rows and columns. 
While the total number of rows is $d^n$, it only has $\binom{d+n-1}{n}$ unique rows corresponding to $[d]^n$ up to the action of $\Pi^n$.
Similarly, it has only $\binom{d+m-n-1}{m-n}$ unique columns.
We calculate the SVD only on the subset of unique rows and unique columns. We rescale each of the unique rows by the square root of the number of times it appears. 
This preserves the dot-product between the columns and rows, so the SVD of $\mat(\TT)$ is recovered from the SVD of the submatrix.
The submatrix has approximately $\frac1{n!}$ of the rows and $\frac1{(m-n)!}$ of the columns of the full matrix, so this speeds up \textsc{Extract Subspace} by about a factor of $n!((m-n)!)^2$ (assuming $n\le m - n$, which is often the case).%

\subsubsection{Approximately low-rank tensors}

We explain how to tweak SPM to deal with approximately low-rank tensors, which often arise in applications. 

First, for approximately low-rank tensors, the flattening is not exactly low-rank. Nevertheless, it is an approximately low-rank matrix, therefore we may use SVD to obtain the best low-rank approximation, and select the rank using explained variance or other PCA techniques.

As for the \textsc{Power Method} routine, the noise leads to tensor components having a function value less than $1$. Nevertheless, in practice we see that the function value of the maxima are still close to $1$, and therefore we add an hyperparameter $\zeta$ (close to $1$) such that if the \textsc{Power Method} converges to a point with function value $>\zeta$, we accept it as a tensor component.  If an estimate of the noise of the tensor is available, it may be used to inform the choice of $\zeta$; see \cite{kileel2021landscape} for  theoretical analysis related to this.

Finally, regarding \textsc{Deflation}, we note that for exactly low rank tensors, we have that $\opvec(\a_{i}^{\otimes n})\in \colspan(\Mx{U})$, which implies $\Nrm{\Vc\alpha}=1$. For dealing with approximately low-rank tensors where $\Nrm{\Vc\alpha}<1$, we modify the deflation formula to:
\begin{equation}
\lambda_i \gets  \frac{\Nrm{\Vc{\beta}}\Nrm{\Vc{\alpha}}}{{\Vc{\beta}}^{\Tr} \Mx{C} {\Vc{\alpha}}}.
\end{equation}

\subsubsection{Hyperparameters}

The implementation of SPM uses hyperparameters:
\begin{itemize}
\item $\kappa$, the distance between \textsc{Power Method} iterates below which we declare \textsc{Power Method} has converged (\textit{default $= 1 \cdot 10^{-14}$});
\item $\zeta$, the minimum function value for $F_{\mathcal{A}}$ for which we accept the \textsc{Power Method} as having converged to a CP component (\textit{default }$= 0.99$); %
\item the maximum number of \textsc{Power Method} iterations (\textit{default }$= 5000$); 
\item the maximum number of repetitions of the overall \textsc{Power Method} step, if the function values are below $\zeta$, after which we pick the iterate with the biggest function value (\textit{default }$= 3$).%
\end{itemize}

% --- End Inserted File ---

% ---- Inserted File ----
\section{Power Method Analysis}\label{sec:tdpowermethod}

The object of this section is to 
prove the following convergence guarantee for \textsc{Power Method}, which is the only part of SPM using nonconvex optimization.  %
\begin{theorem}\label{thm:power}
Assume $m\ge 3$ and let $\TT \in \cS_\dimL^{m}$ satisfy a CP decomposition 
\eqref{eq:tensordecdef}, where $\lambda_i \in \R$ and  $\a_{i} \in \mathbb{S}^{\dimL-1}$ for $i = 1, \ldots, r$.
Let $\cA = \Span\{ \a_{1}^{\otimes n}, \ldots, \a_{r}^{\otimes n} \} \, \subset \, \cS_\dimL^n$ as in \eqref{eq:def_cA}, where $n = \lceil \frac{m}{2} \rceil$.
Set $F_{\cA}(\x) = \| P_{\cA}(\x^{\otimes n})\|^2$ as in \eqref{eq:def-FcA}, and consider the constrained 
optimization problem:
\begin{equation} \label{eq:optim}
\max_{\x\in \mathbb{S}^{\dimL-1}} \,\,\, F_{\cA}(\x).%
\end{equation}
Following \textsc{Power Method}, define the sequence:
\begin{equation} \label{eq:iter}
\x_{k+1} = \frac{\dotp{P_{\cA}(\x_k^{\otimes n})}{\x_k^{\otimes(n-1)}} + \gamma \x_k} {\|\dotp{P_{\cA}(\x_k^{\otimes n})}{\x_k^{\otimes(n-1)}} + \gamma \x_k\|}.
\end{equation}

Here $\x_1 \in \mathbb{S}^{\dimL-1}$ is an initialization, and $\gamma \in \R_{> 0}$ is a fixed shift such that $F_{\cA}(\x) + \gamma (\x^{\Tr} \x)^n$ is a strictly convex function on $\R^n$.  For example, $\gamma >  \sqrt{\frac{n-1}{n}}$ is a sufficiently large shift. Then
\begin{itemize}\setlength\itemsep{0.45em}
\item For all initializations $\x_1$,  \eqref{eq:iter} is well-defined and converges monotonically to a first-order critical point $\x_*$ of \eqref{eq:optim} at no less than an algebraic rate.
That is, $F_{\cA}(\x_{k+1}) \geq F_{\cA}(\x_k)$ for all $k$, and there exist constants $\tau = \tau(\mathcal{A}, \gamma, \x_1) > 1$ and $C = C(\mathcal{A}, \gamma, \x_1) > 0$ such that $\|\x_k - \x_*\| \leq C k^{- \tau}$ for  all $k$. 
\item For a full Lebesgue-measure subset of initializations $\x_1$, \eqref{eq:iter} converges to a second-order critical point of \eqref{eq:optim}.
\item If $r \leq \binom{\dimL+n-1}{n} - \dimL$ and $\a_1, \ldots, \a_r$ are Zariski-generic, the global maxima of \eqref{eq:optim} are precisely $\pm \a_{i}$.  If $r \leq \binom{\dimL+n-1}{n} - \dimL + 1$ and $\a_1, \ldots, \a_r$ are Zariski-generic, each $\pm a_i$ is locally attractive: for all initializations $\x_1$ sufficiently close to $\pm \a_{i}$, \eqref{eq:iter} converges to $\pm \a_{i}$ at no less than an exponential rate.  That is, there exist positive constants $\delta = \delta(\mathcal{A}, \gamma, i)$, 
$\tau = \tau(\mathcal{A}, \gamma, i)$ and $C = C(\mathcal{A}, \gamma, i)$ such that $\| \x_1 - \pm \a_{i} \| \leq \delta$  implies $\| \x_k  - \pm \a_{i}\| \leq C e^{-k\tau}$ for all $k$.
\end{itemize}
\end{theorem}

\begin{remark}
It is easy to see that if $\gamma > 0$ then the \textsc{Power Method} sequence \eqref{eq:iter} is well-defined.
The denominator in  \eqref{eq:iter} does not vanish, because 
$\langle P_{\cA}(\x_k^{\otimes n}) \cdot \x_k^{\otimes (n-1)} + \gamma \x_k, \x_k \rangle = \| P_{\cA}(\x_k^{\otimes n}) \|^2  \, + \, \gamma > 0$.
\end{remark}

\begin{remark} \label{rem:lagrange}
In \cref{thm:power}, critical points are understood in the usual sense of manifold optimization.  That is, $\x_* \in \mathbb{S}^{d-1}$ is first-order critical if 
$\rgrad F_{\cA}(\x_*) = 0$, 
and it is second-order critical if in addition
$\rhess F_{\cA}(\x_*) \preceq 0$,
where grad and Hess denote the \textit{Riemannian gradient} and \textit{Riemannian Hessian} on $\mathbb{S}^{d-1}$ respectively (see \cite[Prop.~4.6]{boumal2023introduction} and \cite[Prop.~6.3]{boumal2023introduction}).  More concretely, from \cite[Sec.~4.2]{absil2013extrinsic} it holds
\begin{align}
&\rgrad F_{\cA}(\x_*) = (\Id - \x_* \x_*^{\Tr}) \nabla F_{\cA}(\x_*), \label{eq:def-riem-grad} \\  &\rhess F_{\cA}(\x_*)  = (\Id - \x_* \x_*^{\Tr}) \nabla^2 F_{\cA}(\x_*) (\Id - \x_* \x_*^{\Tr}) - (\x_*^{\Tr} \nabla F_{\cA}(\x_*)) (\Id - \x_* \x_*^{\Tr}),  \label{eq:def-riem-hess}
\end{align}
where $\nabla$ and $\nabla^2$ denote the Euclidean gradient and Hessian respectively.
\end{remark}

The proof of \cref{thm:power} spans the remainder of \cref{sec:tdpowermethod}, with  details appearing in the appendices.  
As an outline, we identify the iteration \eqref{eq:iter} with SS-HOPM \cite{kolda2011shifted} for computing tensor Z-eigenvectors, except that now SS-HOPM is 
applied to a certain modification of the tensor $\TT$ denoted $\TT[\widetilde]$.
Then we prove the first two bullets of \cref{thm:power} by a new-and-improved general analysis of SS-HOPM. %
For the third bullet we linearize  \eqref{eq:iter} and make a geometric argument exploiting properties of $\TT[\widetilde]$.

\subsection{Connection with SS-HOPM} \label{eq:connection-sshopm}
In \cref{alg:estd}, let $\U_1,\dots,\U_r \in \cS_\dimL^n$ be the orthonormal basis of $\cA \subset \cS_\dimL^n$ given by the columns of $\Mx{U}$. 
\Cref{eq:PAdef2-first} implies 
\begin{equation*}\label{eq:PAdef2}
F_{\cA}(\x) = \|P_\cA (\x^{\otimes n})\|^2 = \Ang{P_\cA (\x^{\otimes n}), \x^{\otimes n} } = \sum_{i = 1}^r  \dotpb{\U_i}{ \x^{\otimes n}}^2.
\end{equation*}
We define the even-order tensor
\begin{equation} \label{eq:Tmodify}
\TT[\widetilde] = \sum_{i = 1}^r  \Sym(\U_i \otimes \U_i) \in \cS_\dimL^{2n}.
\end{equation}
Notice that by \cref{lemma:inner_tensor_powers,lemma:sym_orthogonal_projection}, 
\begin{equation*}\label{eq:PAdef4}
	\langle \TT[\widetilde], \sop[2n]{\x} \rangle \,\,=\,\, \sum_{i = 1}^r  \dotpb{\U_i}{ x^{\otimes n}}^2 = F_{\cA}(\x),
\end{equation*}
and
\begin{equation*}\label{eq:PAdef3a}
\TT[\widetilde]\cdot \x^{\otimes (2n-1)} \,\,=\,\, \frac{1}{2n}\nabla F_{\cA}(\x) = \sum_{i = 1}^r  \dotpb{\U_i}{ \x^{\otimes n}}\U_i \cdot \x^{\otimes (n-1)} \,\,=\,\,  \dotp{P_\cA(\x^{\otimes n})}{\x^{\otimes (n-1)}}.
\end{equation*}
It follows that the \textsc{Power Method} iterations \eqref{eq:powermethodit} and \eqref{eq:iter} for $\TT$ coincide with SS-HOPM iterations applied to the tensor $\TT[\widetilde]$, with shift $\gamma$. 
Moreover, from a tensor standpoint such iterations make sense:  \cref{prop:program_global_maxima} implies that the CP components $\a_i$ of $\TT$ are Z-eigenvectors of $\TT[\widetilde]$, and SS-HOPM computes Z-eigenvectors.

In  \cite{kolda2011shifted} on SS-HOPM, the shift  is chosen so  the corresponding homogeneous polynomial becomes a convex function on $\mathbb{R}^d$.  Here $\TT[\widetilde]$ corresponds to  the polynomial $F_{\cA}(\x)$, and to $F_{\cA}(\x) + \gamma (\x^{\Tr} \x)^{n}$ with the shift. So the next lemma lets us choose $\gamma$.

\begin{lemma} \label{lem:shift}
	Let $\cA \subset \cS_\dimL^n$ be any linear subspace (spanned by rank-1 points or not).
 Let
	$P_{\cA} : \cS_\dimL^n \rightarrow \mathcal{A}$ be orthogonal projection onto $\cA$, $F_{\cA}(\x) = \| P_{\cA}(\x^{\otimes n}) \|^2$ and $\nu\in [0,1]$.
	If $\|\x\| = 1$ and $F_{\cA}(\x)\ge \nu$, then 
	\begin{equation} \label{eq:hessian_psd}
		\frac1{2n}\,\, \min_{\y\in \mathbb{S}^{\dimL-1}}\,\,  \y^{\Tr}\, \nabla^2 F_{\cA}(\x) \,\y \,\, \ge \,\, - \, \sqrt{\frac{n-1}{n}}h(\nu),
	\end{equation}
	where
	\begin{equation}\label{eq:gamma_k}
		h(\nu)=\begin{cases}
			1-\frac{\nu}{2}&\text{if }\nu \le \frac{2}{3}\\
			\sqrt{2\nu(1-\nu)}&\text{if }\nu > \frac{2}{3}.
		\end{cases}
	\end{equation}
	In particular, $F_{\cA}(\x) + {\gamma} (\x^{\Tr} \x)^{n}$ is strictly convex on $\R^{\dimL}$ 
	whenever $\gamma > \sqrt{\frac{n-1}{n}}$.
\end{lemma}

\Cref{lem:shift} is proven in \cref{app:proofshift} by a lengthy but direct calculation.  %

\subsection{Global convergence of SS-HOPM} \label{subsec:sharpen}
Here we sharpen the analysis of SS-HOPM in general.  
In this subsection only,
$F(\x)$ stands for \textit{any} homogeneous polynomial 
function on $\R^\dimL$ of degree $2n$, that is, not necessarily of the form $F_{\cA}(\x)$ for some subspace $\cA$.  It corresponds to an \textit{arbitrary} symmetric tensor  $\TT[\widetilde] \in \cS^{2n}_d$ through $F(\x) = \langle \TT[\widetilde], \x^{\otimes 2n}\rangle$, rather than  $\TT[\widetilde]$ as in \eqref{eq:Tmodify}.
We consider the optimization problem
\begin{equation} \label{eq:maxf}
\max_{\x \in \mathbb{S}^{\dimL-1}} \, F(\x),
\end{equation}
whose critical points are the Z-eigenvectors of $\TT$.  Like in \eqref{eq:iter}, SS-HOPM follows the sequence 
\begin{equation} \label{eq:iter1}
\x_{k+1} = \frac{ \tfrac{1}{2n} \nabla F(\x_k)  +  \gamma  \x_k} {\| \tfrac{1}{2n} \nabla F(\x_k)  +  \gamma  \x_k \|},
\end{equation}
where $\x_1 \in \mathbb{S}^{\dimL-1}$ is an initialization and $\gamma \in \R$ is a fixed shift.  Assume $\gamma$ is chosen so
\begin{equation}\label{eq:shift-convex-1}
G(\x) = F(\x) + {\gamma} (\x^{\Tr} \x)^{n} \,\, \text{is a strictly convex function on } \mathbb{R}^d.
\end{equation}
For $\x \in \mathbb{S}^{d-1}$ denote
\begin{equation} \label{eqn:G}
\Psi(\x) = \frac{\nabla G(\x)}{\| \nabla G(\x) \|},
\end{equation}
so that \eqref{eq:iter1} may be written 
\begin{equation} \label{eq:max_F}
\x_{k+1} = \Psi(\x_k).
\end{equation}
Note that $\x_* \in \mathbb{S}^{d-1}$ is first-order critical for \eqref{eq:maxf} if and only if $\x_*$ is a fixed point \nolinebreak of \nolinebreak $\Psi$. %

In the next result we resolve Kolda-Mayo's conjecture \cite[p.~1107]{kolda2011shifted} for even-order tensors, corresponding to even-degree homogeneous polynomials, by establishing that SS-HOPM converges for all initializations.  
A main tool comes from \L ojasiewicz's  inequality for real analytic functions \cite{lojasiewicz1965ensembles}.

\begin{theorem}[Unconditional convergence of SS-HOPM] \label{thm:always_converge}
Assume the setting of \eqref{eq:maxf}-\eqref{eq:max_F}. 
Then for all initializations $\x_1 \in \mathbb{S}^{d-1}$, the sequence \eqref{eq:iter1} is well-defined and \eqref{eq:iter1} converges monotonically to a first-order critical point of \eqref{eq:maxf} at no less than an algebraic rate.
\end{theorem}
\begin{proof}
Note the denominators in \eqref{eq:iter1} do not vanish, %
as $\nabla G(\x_k) = \nabla F(\x_k) + 2n \gamma (\x_k^{\Tr} \x_k)^{n-1} \x_k =  \nabla F(\x_k) + 2n\gamma \x_k$, so $\langle \x_k, \tfrac{1}{2n} \nabla F(\x_k) + \gamma \x_k \rangle = \langle \x_k, \tfrac{1}{2n} \nabla  G(\x_k) \rangle = \tfrac{1}{2n} G(\x_k)$ which is positive since $G$ is even and strictly convex. %
It follows that \eqref{eq:iter1} is well-defined.
Also clearly the constrained critical points of $F$ and $G$ coincide, because $G(\x) = F(\x) + \gamma$ for $\x \in \mathbb{S}^{d-1}$, so we may show $(\x_{k})_{k=1}^{\infty}$ 
converges monotonically at (at least) a power rate to a first-order critical point of \begin{equation} \label{eq:maxf2}
\max_{\x \in \mathbb{S}^{\dimL-1}} \, G(\x),
\end{equation}.  

By convexity of $G$,
\begin{equation} \label{eq:G-convexity}
G(\x_{k+1}) - G(\x_k) \geq \nabla G(\x_k)^{\Tr} (\x_{k+1} - \x_k).
\end{equation}
From \eqref{eq:max_F} and the Cauchy-Schwarz inequality, 
\begin{equation*}
\nabla G(\x_k)^{\Tr} (\x_{k+1} - \x_k) = \| \nabla G(\x_k)\| - \nabla G(\x_k)^{\Tr} \x_k \geq \| \nabla G(\x_k)\| - \|\nabla G(\x_k)\| \|\x_k\| = 0.
\end{equation*}
Thus $(G(\x_k))_{k=1}^{\infty}$ monotonically increases.

Next suppose $(\x_k)_{k=1}^{\infty}$ indeed converges to $\x_* \in \mathbb{S}^{d-1}$.
Taking limits in \eqref{eq:max_F}, continuity of $G$ implies
\begin{equation*}
\x_* = \frac{\nabla G(\x_*)}{\| \nabla G(\x_*) \|}.
\end{equation*}
In particular, 
$(\Id - \x_* \x_*^{\Tr}) \nabla G(\x_*) = 0$, and so $\x_*$ is a first-order critical point of \eqref{eq:maxf2}  (see \cref{rem:lagrange}).

It remains to show that $(\x_k)_{k=1}^{\infty}$ actually converges, and that it does so at at least a power rate.
To this end we apply a convergence result of Schneider and Uschmajew \cite[Theorem~2.3]{schneider2015convergence} based on the \L ojasiewicz inequality for real analytic functions, see 
a precise statement in \cref{app:convergence-result}.
We take $\mathcal{M}$ in \cref{thm:meta-converge} to be $\mathbb{S}^{d-1} \subset \mathbb{R}^d$.

To verify condition (\textbf{A1}) in \cref{thm:meta-converge}, we must verify there exists $\sigma > 0$ such that for large enough $k$, 
\begin{equation} \label{eq:firstCond}
G(\x_{k+1}) - G(\x_{k}) \geq \sigma \|\textup{grad } G(\x_k) \| \|\x_{k+1} - \x_k \|,
    \end{equation}
    where $\textup{grad } G(\x_k) = (\Id - \x_k\x_k^{\Tr})\nabla G(\x_k)$ denotes the Riemannian gradient. 
    In fact, $\sigma = \frac{1}{2}$ works.
    Indeed from \eqref{eq:G-convexity} and \eqref{eq:max_F},  
\begin{align}\setlength\itemsep{0.2em} \label{eq:loj}
G(\x_{k+1})-G(\x_k) \,
&\ge \, \nabla G(\x_k)^{\Tr}(\x_{k+1}-\x_k) 
\, = \, \|\nabla G(\x_k)\| \x_{k+1}^{\Tr}(\x_{k+1}-\x_k) \nonumber \\
 & = \, \|\nabla G(\x_k)\| (1-\left<\x_{k+1},\x_k\right>)
\, = \, \tfrac{1}{2}\|\nabla G(\x_k)\| \|\x_{k+1}-\x_k\|^2.
\end{align}
On the other hand,
\begin{align}\setlength\itemsep{0.2em}
&\|
\rgrad  G(\x_k)\|^2 
\, = \, \|(\Id - \x_k\x_k^{\Tr})\nabla G(\x_k)\|^2 
\, = \, \|\nabla G(\x_k)\|^2\|\x_{k+1}-\left<\x_{k+1},\x_k\right>\x_k \|^2 \nonumber \\
&\, = \, \|\nabla G(\x_k)\|^2 (1-\left<\x_{k+1},\x_k\right>^2) 
\, = \, \|\nabla G(\x_k)\|^2 (1-\left<\x_{k+1},\x_k\right>)(1+\left<\x_{k+1},\x_k\right>) \nonumber \\
& \, \le \, 2 \|\nabla G(\x_k)\|^2 (1-\left<\x_{k+1},\x_k\right>) 
\, = \, \|\nabla G(\x_k)\|^2 \|\x_{k+1}-\x_k\|^2. \nonumber
\end{align}
Substituting the square root of this into \eqref{eq:loj} yields \eqref{eq:firstCond} with $\sigma = \tfrac{1}{2}$.

 To check condition (\textbf{A2}) in \cref{thm:meta-converge}, it is required to verify  if $k$ is large enough then 
  $
\rgrad  G(\x_k) = 0$ implies  $\x_{k+1} = \x_k$.
Assume $
\rgrad  G(\x_k) = 0$.  Since $
\rgrad  G(\x_k) = (\Id - \x_k\x_k^{\Tr})\nabla G(\x_k)$, it follows $\nabla G(\x_k)$ is parallel to $\x_k$.
As explained in the first sentence of the proof, $\langle  \x_k, \tfrac{1}{2n} \nabla G(\x_k)\rangle > 0$.  So, $\nabla G(\x_k)$ is a positive multiple of $\x_{k+1}$.  By \eqref{eq:max_F}, $\x_{k+1} = \x_k$.

For condition (\textbf{A3}) in \cref{thm:meta-converge}, we must verify there exists a constant $\rho > 0$ such that for large enough $k$ it holds
$\|\x_{k+1} - \x_{k} \| \geq \rho \|\textup{grad } G(\x_k) \|$.
However by the above, we may take $\rho = \left( \max_{\| \x \| = 1} \| \nabla G(\x)  \| \right)^{-1}$.

\Cref{thm:meta-converge} implies the sequence converges at at least an algebraic rate.  
\end{proof}

Next we prove that SS-HOPM converges to second-order critical points of \eqref{eq:maxf} for almost all initializations.  In the language of \cite{kolda2011shifted}, SS-HOPM converges to {stable eigenvectors}. 
We adopt the proof strategy of \cite{panageas2016gradient, lee2016gradient} based on the center-stable manifold theorem from dynamical systems.  The following lemma is a key calculation.
\begin{lemma} \label{lem:calc}
Assume the setting of \cref{eq:maxf,eq:iter1,eq:shift-convex-1,eqn:G,eq:max_F}. 
For all $\x \in \mathbb{S}^{d-1}$ the Jacobian $D \Psi (\x)$ as a linear map between tangent spaces to the sphere is
\begin{equation} \label{eq:jac-generalx}
D \Psi (\x) \, = \, (\Id - \Psi(\x) \Psi(\x)^{\Tr}) \frac{\nabla^2 G(\x)}{\| \nabla G(\x) \|} (\Id - \x \x^{\Tr}),
\end{equation}
where $\nabla^2 G(\x)$ denotes the Euclidean Hessian.  At a first-order critical point $\x_* \in \mathbb{S}^{\dimL-1}$ of \eqref{eq:maxf} it holds 
\begin{equation} \label{eq:jac-firstx}
D\Psi(\x_*) \,\, = \,\, \frac{\rhess F(\x_*)}{2n (F(\x_*) + \gamma)} \,  + \, \, \Id - \x_* \x_*^{\Tr},
\end{equation}
where $\rhess$ denotes the Riemannian Hessian.
\end{lemma}

The proof of \cref{lem:calc} is given in \cref{app:proofcalc}. It implies two additional lemmas, recorded next.

\begin{lemma}\label{eq:local-diffeo-2}
Assume the setting of \cref{eq:maxf,eq:iter1,eq:shift-convex-1,eqn:G,eq:max_F}. 
Then $\Psi$ is a local diffeomorphism from $\mathbb{S}^{d-1}$ to $\mathbb{S}^{d-1}$.
\end{lemma}
\begin{proof}
It suffices to show that $D\Psi(\mathbf{x})$ is an isomorphism on tangent spaces for all $\x \in \mathbb{S}^{d-1}$.  From \eqref{eq:jac-generalx} we need to show that if $\langle \y, \x \rangle = 0$ and $\y \neq 0$ then $\nabla^2 G(\x) \y$ is not parallel to $\Psi(\x)$.  But this follows from the facts that $\nabla^2 G(\x)$ is nonsingular (since $G$ is strictly convex), and $\nabla^2 G(\x) \x = (2n-1) \nabla G(\x)$ is parallel to $\Psi(\x)$.
\end{proof}

\begin{lemma} \label{lemma:center_stable_manifold}
Assume the setting of \cref{eq:maxf,eq:iter1,eq:shift-convex-1,eqn:G,eq:max_F}. 
Let $\x_* \in \mathbb{S}^{\dimL-1}$ be a first-order critical point of \eqref{eq:maxf} but not a second-order point. 
 Then there exist an open neighborhood $B_{\x_*} \subset \mathbb{S}^{\dimL-1}$ of $\x_*$ and 
	a smoothly embedded disk $D_{\x_*}$ containing $\x_*$ of dimension strictly less than $\dimL-1$ satisfying
	\begin{equation} \label{eq:helpful-cs}
		\left( \Psi^{k}(\x) \in B_{\x_*}   \, \forall \, k \geq 0 \right) \Rightarrow \x \in D_{\x_*}.
	\end{equation}
\end{lemma}
\begin{proof}
By assumption, $\rhess F(\x_*)$ has an eigenvalue which is strictly positive.  Then \eqref{eq:jac-firstx} implies $D\Psi(\x_*)$ has an eigenvalue that exceeds $1$, using $F(\x_*) + \gamma  = G(\x_*) > 0$ from strict convexity of $G$.  The conclusion is now immediate from the center-stable manifold theorem; see \cref{thm:CMT} in which we take $\mathcal{M} = \mathbb{S}^{d-1}$.
\end{proof}

We have all ingredients to prove that SS-HOPM almost always converges to stable eigenvectors.
\begin{theorem}[Almost always convergence of SS-HOPM to second-order critical point] \label{thm:noSaddles}
Assume the setting of 
\cref{eq:maxf,eq:iter1,eq:shift-convex-1,eqn:G,eq:max_F}. 
Then there is a Lebesgue-measure zero subset $\Omega \subset \mathbb{S}^{d-1}$ such that for all  initializations $\x_1 \in \mathbb{S}^{d-1} \setminus \Omega$, the sequence \eqref{eq:iter1} converges to a second-order critical point of \eqref{eq:maxf}.
\end{theorem}
\begin{proof}
Let $\mathcal{C} \subset \mathbb{S}^{d-1}$ denote the set of ``bad'' critical points, i.e., first-order but not second-order critical points $\x_*$ of \eqref{eq:maxf}.  Let $\mathcal{I} \subset \mathbb{S}^{d-1}$ denote the set of ``bad'' initializations, i.e., $\x_1$ such that \eqref{eq:iter1} converges to an element of $\mathcal{C}$.  By \cref{thm:always_converge}, we know that for all initializations \eqref{eq:iter1} converges to a first-order critical point of \eqref{eq:maxf}.  Therefore it suffices to prove that $\mathcal{I}$ is measure zero in $\mathbb{S}^{d-1}$.

For each $\x_* \in \mathcal{C}$, by Lemma~\ref{lemma:center_stable_manifold} there exist 
 an open neighborhood $B_{\x_*} \subset \mathbb{S}^{\dimL-1}$ of $\x_*$ and 
	a smoothly embedded disk $D_{\x_*} \subset \mathbb{S}^{d-1}$ containing $\x_*$ of dimension strictly less than $\dimL-1$ satisfying \eqref{eq:helpful-cs}.
By second countability of $\mathbb{S}^{d-1}$, there exists a countable subset $\mathcal{C}' \subset \mathcal{C}$ such that 
\begin{equation} \label{eq:myLindelof}
\bigcup_{\x_* \in \mathcal{C'}}  B_{\x_*} \,\, = \,\, \bigcup_{\x_* \in \mathcal{C}}  B_{\x_*},
\end{equation}
see the Lindel\"of property  \cite[p.~191]{munkres}.

Consider $\x_1 \in \mathcal{I}$. There exist $\x_* \in \mathcal{C}'$ and $K \in \mathbb{N}$ such that $\Psi^{k}(\x_1) \in B_{\x_*}$ for all $k \geq K$; indeed, take $\x_* \in \mathcal{C'}$ satisfying $\lim_{k \rightarrow \infty} \Psi^{k}(\x_1) \in B_{\x_*}$ and use that $B_{\x_*}$ is open.  
By \eqref{eq:helpful-cs}, this implies $\Psi^{K}(\x_1) \in D_{\x_*}$, or $\x_1 \in \Psi^{-K}(D_{\x_*})$ where $\Psi^{-K}$ denotes preimage via $\Psi^{K}$. Ranging over $\x_*$ and $K$, it follows 
\begin{equation} \label{eq:big-union}
\mathcal{I} \,\, \subset \, \bigcup_{\x_* \in \mathcal{C}'} \,\, \bigcup_{K \geq 0} \,\, \Psi^{-K}(D_{\x_*}).
\end{equation}

Note that each $D_{\x_*}$ has measure zero in $\mathbb{S}^{d-1}$ since its dimension is strictly less than $d-1$.  Further,  each  $\Psi^{-k}(D_{\x_*})$ has measure zero because $\Psi^{k}$ is a local diffeomorphism by \cref{eq:local-diffeo-2}.  Therefore the right-hand side in \eqref{eq:big-union} has measure zero, being a countable union of measure zero subsets.
Hence $\mathcal{I}$ has measure zero too.  %
\end{proof}

The results on SS-HOPM from this subsection may be of independent interest.  

\subsection{Local linear convergence to $\pm \a_i$} \label{subsec:SSHOPM-connect}%
We return to SPM specifically, and prove the local attractiveness claim in \cref{thm:power}.

\begin{theorem}\label{thm:local_linear_convergence}
Let $r \leq \binom{d+n-1}{d} - d +1$, and $\a_1, \ldots, \a_r \in \mathbb{S}^{d-1}$ be Zariski-generic.  Then the \textsc{Power Method} \eqref{eq:iter} has local linear convergence to each of the $\pm \a_i$'s.
\end{theorem}
\begin{proof}
Denote \eqref{eq:iter} as $\x_{k+1} = \Psi(\x_{k})$ like in \cref{subsec:sharpen}, but with $F = F_{\cA}$ where $\cA = \operatorname{span} \{\a_1^{\otimes n}, \ldots, \a_r^{\otimes n}\}$.  
Clearly $\pm \a_i$ are fixed points of $\Psi$, thus by \cite[p.~18]{rheinboldt1974methods} we need to show that $D\Psi(\pm \a_i)$ has spectral norm strictly less than $1$.  For notational ease, we just consider $\a_i$ as the result for $-\a_i$ will be immediate by evenness of $F_{\cA}$.

From \eqref{eq:jac-firstx} and \eqref{eq:def-riem-hess},
\begin{equation} \label{eq:Dphi-locallinear}
D\Psi(\a_i) \,\, = \,\, \frac{1}{2n(1+\gamma)} \big{(}   (\Id - \a_i \a_i^{\Tr}) \nabla^2 F_{\cA}(\a_i) (\Id - \a_i \a_i^{\Tr}) \, + \, 2n\gamma (\Id - \a_i \a_i^{\Tr}) 
\big{)},
\end{equation}
where we used $F_{\cA}(\a_i) = 1$ and $\a_i^{\Tr} \nabla F_{\cA}(\a_i) = 2n F_{\cA}(\a_i) = 2n$.  Notice that $D \Psi(\a_i)$ is self-adjoint.  Thus to bound its spectral norm, we can show that for all $\y$ with $\langle \y, \a_i \rangle = 0$ and $\| \y \|=1$ it holds  $\big{|} \y^{\Tr} D \Psi(\a_i) \y \big{|} < 1$.  Inserting \eqref{eq:Hexpansion} into \eqref{eq:Dphi-locallinear}, 
\begin{align}
\y^{\Tr} D \Psi(\a_i) \y &= \frac{1}{2n(1+\gamma)} \big{(} \y^{\Tr} \nabla^2 F_{\cA}(\a_i) \y +  2n\gamma \big{)}\nonumber
\\[0.05em]
&= \frac{1}{1+\gamma} \big{(} n \| P_{\cA}(\a_i^{\otimes (n-1)} \otimes \y) \|^2 + (n-1) \langle P_{\cA}(\a_i^{\otimes n}), \a_i^{\otimes (n-2)} \otimes \y^{\otimes 2}\rangle + \gamma \big{)} \nonumber \\[0.05em]
&= \frac{1}{1+\gamma} \big{(} n \| P_{\cA}(\a_i^{\otimes (n-1)} \otimes \y) \|^2 + (n-1) \langle \a_i^{\otimes n}, \a_i^{\otimes (n-2)} \otimes \y^{\otimes 2} \rangle + \gamma \big{)} \nonumber \\[0.05em]
&= \frac{1}{1+\gamma} \big{(} n \| P_{\cA}(\operatorname{sym}(\a_i^{\otimes (n-1)} \otimes \y)) \|^2 + \gamma \big{)}, \nonumber
\end{align}
where we used $P_{\cA}(\a_i^{\otimes n}) = \a_i^{\otimes n}$, 
that $P_{\cA}$ is a projector onto a subspace of symmetric tensors, and 
$\langle \a_i^{\otimes n}, \a_i^{\otimes (n-2)} \otimes \y^{\otimes 2} \rangle = \langle \a_i, \a_i \rangle^{n-2} \langle \a_i, \y \rangle^2 = 0$ by \cref{lemma:inner_tensor_powers}.
Therefore
\begin{equation*}
\big{|} \y^{\Tr} D\Psi(\a_i) \y \big{|} \leq \frac{1}{1+\gamma}\big{(} n \| \operatorname{sym}(\a_i^{\otimes (n-1)} \otimes \y) \|^2 + \gamma \big{)} = 1,
\end{equation*}
with equality if and only if $\operatorname{sym}(\a_i^{\otimes (n-1)} \otimes \y) \in \cA$.   Thus the theorem follows from the proposition below.
\end{proof}

\begin{proposition}\label{prop:transverse}
Assume $r \leq \binom{d+n-1}{d} - d +1$.  Let $\a_1, \ldots, \a_r \in \mathbb{S}^{d-1}$ be Zariski-generic and $\cA = \operatorname{span}\{ \a_1^{\otimes n}, \ldots, \a_r^{\otimes n}\}$.  Then for each $i$ it holds
$$  \big{\{} \! \operatorname{sym}(\a_i^{\otimes (n-1)} \otimes \y) : \langle \y, \a_i \rangle =0 \big{\}} \, \cap \, \cA  \,\, = \,\, 0.$$ 
\end{proposition}

\begin{proof}
Without loss of generality, $i=1$.  A point in the intersection is of the form
\begin{equation}\label{eq:intersection-my} \alpha_1 \a_1^{\otimes n} + \alpha_2 \a_2^{\otimes n} + \ldots + \alpha_r \a_r^{\otimes n} = \operatorname{sym}(\y \otimes \a_1^{\otimes (n-1)}) 
\end{equation}
for some $\alpha_j \in \mathbb{R}$ and $\y$ with $\langle \y, \a_i \rangle = 0$.  
Write $\pi$ for the projector from $\cT_\dimL^m$ to the orthogonal complement of $\operatorname{sym}(\mathbb{R}^d \otimes \a_1^{\otimes (n-1)})$, and apply it to \eqref{eq:intersection-my} to obtain 
\begin{equation}\label{eq:I-applied-pi}
\alpha_2 \pi(\a_2^{\otimes n}) + \ldots + \alpha_r \pi(\a_r^{\otimes n}) = 0.
\end{equation}
Note that $\pi(\a_2^{\otimes n}), \ldots, \pi(\a_r^{\otimes n})$ 
are $r-1$ generic points in the closure of the projection of the affine cone over the Veronese variety under $\pi$, and that this variety linearly spans its ambient space which has linear dimension 
$\binom{d+n-1}{d}-d$.  Since $r -1 \leq \binom{d+n-1}{d}-d$ by assumption, $\pi(\a_i^{\otimes n})$ are linearly independent.  Hence \eqref{eq:I-applied-pi} implies $\alpha_2 = \dots = \alpha_r = 0$.  
Returning to \eqref{eq:intersection-my}, it follows $\alpha_1 = 0$ and 
$\y = 0$, hence the intersection is zero.
\end{proof}

We remark that \cref{prop:transverse} has a geometrical interpretation.  Namely, when the projectivization of $\cA$ and the Veronese variety are expected to have a zero-dimensional intersection, they in fact intersect \textit{transversely} at each of $\a_i^{\otimes n}$ (see \cite{harris1992algebraic}).

\subsection{{Putting it together}}
To sum up, the convergence analysis is complete. 

\begin{proof}[Proof of \cref{thm:power}]
That $\gamma > \sqrt{\frac{n-1}{n}}$ is a sufficient shift to guarantee strict convexity follows from \cref{lem:shift}.
The first bullet of \cref{thm:power} follows from \cref{thm:always_converge}.  
The second bullet follows from \cref{thm:noSaddles}.
The first sentence of the third bullet is due to \cref{prop:program_global_maxima}, and the rest is by \cref{thm:local_linear_convergence}.
\end{proof}

Beyond the results proven here, in the next section we see empirically that the \textsc{Power Method} is robust to noise (see \cref{sec:noise_stability}) and that it very often converges to global maxima (see \cref{sec:optimization_landscape}).

% --- End Inserted File ---

% ---- Inserted File ----
\section{Numerical Experiments}\label{sec:tdsimulations}

In this section, we present numerical tests of SPM.  We provide comparisons of runtime, accuracy and noise stability against existing state-of-the art algorithms for symmetric CP decomposition. 
We show comparisons with the following algorithms:
\begin{itemize}
\item Fourth-Order Only Blind Identification (FOOBI) \cite{foobi2007}: a polynomial-time algorithm to decompose fourth order tensors with rank $r\le \mathcal{O}(d^2)$. We use the official MATLAB implementation, provided in Tensorlab+ \cite{vervliet2016tensorlab}.
\item CP decomposition through simultaneous diagonalization/generalized eigenvalue decomposition (GEVD) \cite{leurgans1993decomposition}: a polynomial-time algorithm to decompose third order tensors with rank $r\le d$. We use the implementation in Tensorlab \cite{vervliet2016tensorlab}.
\item The method described in the paper \cite{nie2017low}, which we abbreviated as LRSTA. It is a provable algorithm that decomposes any $m$-order tensors with rank $r~\le~\mathcal{O}(d^{\floor{\frac{m-1}2}})$. We use the official MATLAB implementation, available on the author Nie's webpage, with its default parameters.
\item CP decomposition through non-linear least squares (NLS) \cite{vervliet2016}: a non-guaranteed algorithm that 
employs a Gauss-Newton method to solve the non-convex least-squares optimization problem
\begin{equation}\label{eqn:tensorlabopt}
	\textup{argmin}_{\Mx{A}\in \R^{\dimL\times r}} \Big{\|}\TT - \sum_{i=1}^{r}  \a_{i}^{\otimes m}\Big{\|}^2.
\end{equation}
We use the implementation in Tensorlab (\textsf{ccpd\_nls}) \cite{vervliet2016tensorlab}. We set the maximum number of the second-order iterations to $500$, and the gradient and function tolerances (as stop criteria) to $1 \times 10^{-12}$ and $1 \times 10^{-24}$ respectively, in order to obtain $\Mx{A}$ up to an error bounded by $10^{-12}$.
\end{itemize}
The first three methods are considered leading among provable algorithms, while the latter is a leading method among heuristic algorithms.

In SPM, we use the most-square matrix flattening in the \textsc{Extract Subspace} step (i.e., $n = \lceil m/2 \rceil$ in \eqref{eq:mattensor}). 
For the fairest comparisons we provide all methods with the ground-truth rank, although this can be computed by SPM and FOOBI (see \Cref{sec:practical}).  
All numerical experiments are performed on MATLAB on a personal laptop with a Intel\textsuperscript\textregistered\ Core\texttrademark\ i7-7700HQ CPU and 16.0GB of RAM. The implementation of SPM (in MATLAB and Python), and the code to run the experiments, are available at {\texttt{\url{https://www.github.com/joaompereira/SPM}}}.

\begin{remark} \label{rem:adaptive}
For a modest performance benefit, we implement SPM using adaptive shifts $\gamma_k$, rather than a constant step size $\gamma$.  
Similarly to \cite{kolda2014adaptive}, we choose a
smaller shift at $\x_k$ according to how close $F_{\cA}(\x)$ is to being locally convex around $\x_k$.
Specifically, by \cref{lem:shift} we modify \textsc{Power Method} to the following:
\begin{equation} \label{eq:iter3}
\x_{k+1} = \frac{\dotp{P_{\cA}(\x_k^{\otimes n})}{\x_k^{\otimes(n-1)}} + \gamma_k \x_k} {\|\dotp{P_{\cA}(\x_k^{\otimes n})}{\x_k^{\otimes(n-1)}} + \gamma_k \x_k\|} \quad \,\, \textup{where } \,\,  \gamma_k = \sqrt{\frac{n-1}{n}}h(F_{\cA}(\x_k)),
\end{equation} 
with $h$ defined as in \cref{lem:shift}. This leads to a slight improvement, but it doesn't affect the results qualitatively.  
\end{remark}

\subsection{Runtime comparison}\label{sec:comptimecomp}

Here we compare computation times for computing CP decompositions amongst SPM and the other methods.

In \cref{fig:time_order3}, we plot the computation time (in seconds) for computing the CP decomposition of $\TT$ by SPM, LRSTA, GEVD and NLS as a function of $\dimL$. For this experiment, we set $m=3$. FOOBI is omitted from this plot as it does not apply to third-order tensors. For several values of $\dimL$ ranging from $10$ to $300$, we generate several tensors as follows. 
The rank $r$ is set as $d$, and for each $i=1, \ldots, r$ we independently sample a vector $\mathbf{v}$ from a standard multivariate Gaussian distribution, put $\lambda_i = \|\mathbf{v}\|^m$ and $\a_i = \mathbf{v}/\| \mathbf{v}\|$, and then form $\TT$ as in \eqref{eq:tensordecdef}. For each value of $d$ we sample $20$ tensors in this manner, and report the average and 20\%/80\% quantiles of the runtimes of calculating the CP decomposition through each method. Further, we report the frequency of runs when the method obtains the correct tensor decomposition. We say this holds if the error is bounded as follows: $\Nrm{\TT - \TT[\widehat]}/\Nrm{\TT}<10^{-4}$. %

\begin{figure}
	\centering
	\begin{subfigure}[b]{0.48\textwidth}
		\centering
		\includegraphics[width=\textwidth]{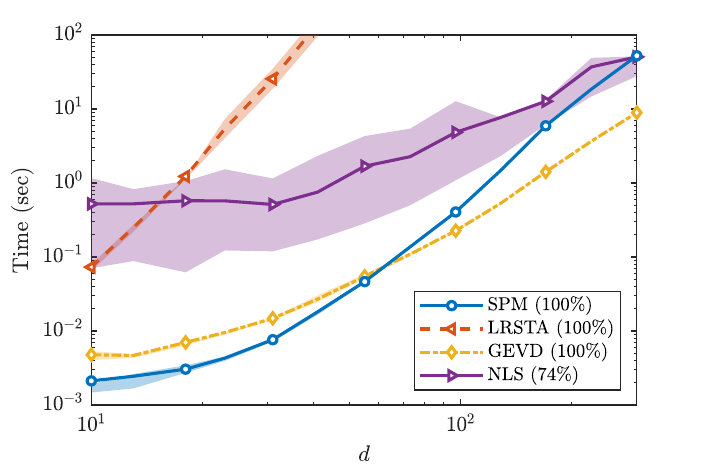}
		\caption{Tensors of  order 3 ($r= d$).}
		\label{fig:time_order3}
	\end{subfigure}
	\hfill
	\begin{subfigure}[b]{0.48\textwidth}
		\centering
		\includegraphics[width=\textwidth]{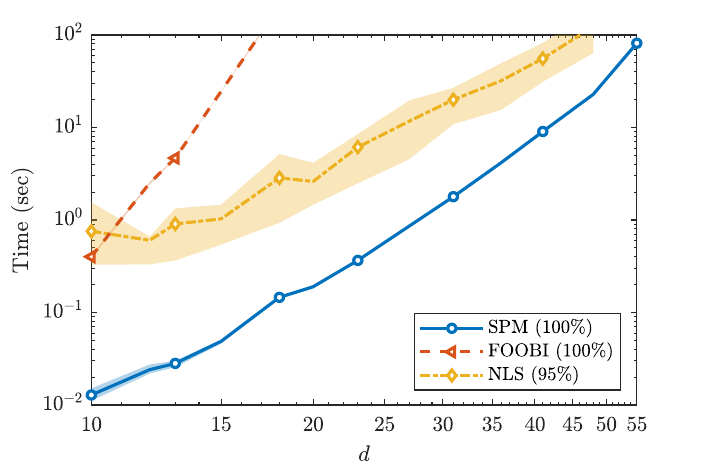}
		\caption{Tensors of  order 4 ($r= \lfloor d^2/3 \rfloor $)}
		\label{fig:time_order4}
	\end{subfigure}
	\caption{Comparison of the computation time between various CP decomposition algorithms. We plot the average computation time as curves, and the shaded areas correspond to the 20\% to 80\% quantiles. The percentages in the legends denote the average fraction of runs each algorithm returned a correct decomposition.}
	\label{fig:three graphs}
\end{figure}
 
In \cref{fig:time_order4}, we show a similar comparison for SPM, FOOBI and NLS. We set $m=4$, with generated $\TT$ as above, but with $\dimL$ ranging from $10$ to $55$ and $r=\left\lfloor \dimL^2/3\right\rfloor$. 
We do not include GEVD and LRSTA in this comparison as they are not able to decompose tensors of order $4$ of such high rank.
Like in the previous comparison, we compute the CP decomposition of 20 tensors by each algorithm and report the average,  standard deviation and  frequency of obtaining the correct tensor decomposition.  

A takeaway from \cref{fig:three graphs} is that SPM is a very competitive algorithm in terms of computation time.
In our experiments it outperforms both FOOBI and NLS for tensors of order 4. As a specific timing example, when $d=48$ the computation time for SPM is 22.8 seconds, while for NLS is 122.9 seconds.
We also observe that, while FOOBI's complexity is polynomial, the exponents are high ($\mathcal{O}(r^4 \dimL^2)$, which is $\mathcal{O}(d^{10})$ if $r = \mathcal{O}(d^2)$).  This makes FOOBI slower than alternatives especially for large fourth-order tensors. 
For tensors of order 3 (\cref{fig:time_order3}), SPM is only outperformed in terms of time by GEVD and its performance is similar to that of NLS for larger tensors.

Finally, note that the computation time of SPM varies less when compared with heuristic algorithms. Although we do not fully understand why this occurs, we believe it is due to the nice behaviour of SPM's optimization landscape. We study the landscape computationally in  \cref{sec:optimization_landscape}, and have theoretical analysis of it in  \nolinebreak \cite{kileel2021landscape}.

\subsection{Noise stability}\label{sec:noise_stability}

Many tensors that arise in real-data applications are not exactly low-rank, and therefore knowing that SPM works for approximately low-rank tensors is important. 
In this subsection, we compare the sensitivity to noise of SPM against the other methods.
For FOOBI and SPM we use the best rank-$r$ approximation of the tensor flattening, computed from the truncated SVD (or eigendecomposition when applicable). 

\begin{figure}
	\centering
	\begin{subfigure}[b]{0.48\textwidth}
		\centering
		\includegraphics[width=\textwidth]{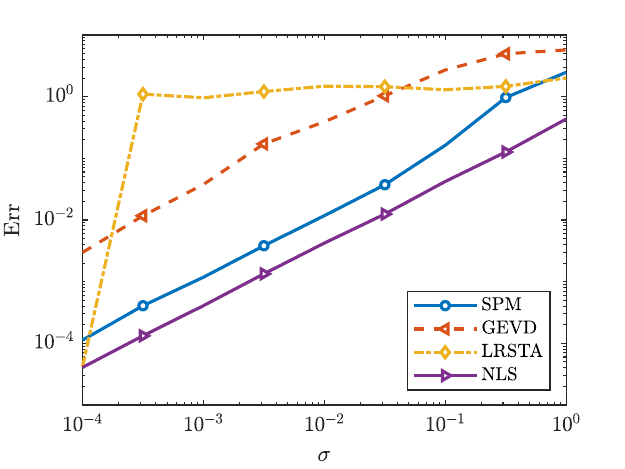}
		\caption{Tensor of order 3, dimension $d=28$ and rank $30$.}
		\label{fig:noise_order3}
	\end{subfigure}
	\hfill
	\begin{subfigure}[b]{0.48\textwidth}
		\centering
		\includegraphics[width=\textwidth]{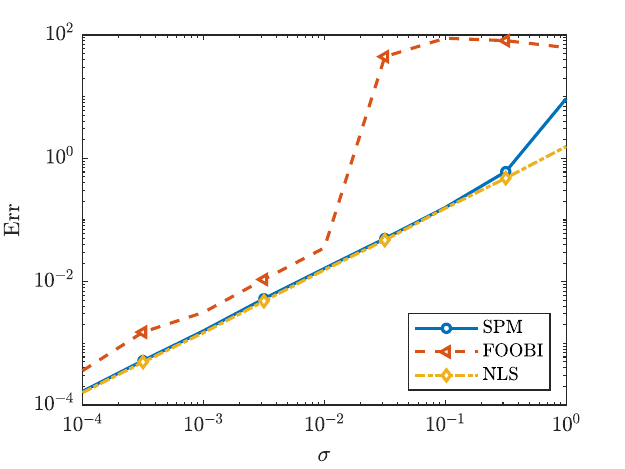}
		\caption{Tensor of order 4, dimension $d=15$ and rank $30$.}
		\label{fig:noise_order4}
	\end{subfigure}
	\caption{Noise stability comparison between various CP decomposition methods. Here, $\sigma$ is the standard deviation of the Gaussian noise added to the entries of tensor that is to be decomposed, and the error is calculated using \eqref{eq:noise_err}.}
\label{fig:SPM_noise}
\end{figure} 

We conduct two experiments comparing SPM's performance on tensors of order 3 and order 4. In these tests, we generate positively correlated CP components as these tensors reveal noise sensitivity better.  Specifically, for a given rank $r$ and for each $i=1, \ldots, r$ we independently sample a vector $\Vc{v}$ from a standard multivariate Gaussian distribution. Then, we set $\Vc[\tilde]{v}=\Vc{v}+\alpha\Vc{1}$, where $\alpha$ is a parameter to be defined later, put $\lambda_i = \|\Vc[\tilde]{v}\|^m$ and $\a_i = \Vc[\tilde]{v}/\|\Vc[\tilde]{v}\|$, and then form a low-rank tensor as in \eqref{eq:tensordecdef}. As defined, the expected correlation between two vectors is $\alpha^2/(1+\alpha^2)$. We set $\alpha=1$ as to have an average correlation of $1/2$. Finally, we add centered Gaussian noise with variance $\sigma^2$ to the entries of the low-rank tensor (in a symmetric manner); this produces an approximately low-rank tensor  to be decomposed by the different algorithms.

In the first experiment, we set $m=3$, $d=30$, $r=28$, and compare SPM with NLS, LRSTA and GEVD, and in the other we set $m=4$, $d=15$, $r=30$, and compare SPM with NLS and FOOBI. In \cref{fig:SPM_noise}, we plot for each algorithm the error, measured using the following formula
\begin{equation}\label{eq:noise_err}
	\text{Err}:= \left\| \TT - \TT[\widehat]\right\| \quad\text{where}\quad \TT = \sum_{i=1}^r \lambda_i \sop{\Vc{a}_i},\quad \TT[\widehat] = \sum_{i=1}^r \hat \lambda_i \sop{\Vc[\hat]{a}_i},
\end{equation}
where $\hat \lambda_i$ and $\sop{\Vc[\hat]{a}_i}$ are the tensor coefficients and tensor components obtained by each algorithm. 

As can be observed in \cref{fig:SPM_noise}, NLS is the most stable to noise, which may be expected since it explicitly minimizes the $\ell^2$-norm. 
However, the sensitivity to noise of SPM is very similar to that of NLS for tensors of order 4, up to the noise level $10^{-1}$. 
For larger noise levels, SPM is less accurate.
We suspect that this is related to a spiked eigenvalue model for the flattening: at that noise level, the eigenvectors of the flattening coming from noise start to overwhelm eigenvectors of the flattening from the signal. Meanwhile, in this comparison FOOBI appears to be much less stable.

For tensors of order 3, SPM is less stable than NLS. We believe that this is due to the matrix of tensor components being poorly-conditioned, which results in a higher error in the subspace estimation from the flattening step in SPM. Nevertheless, SPM is still considerably more stable than LRSTA and GEVD.

\subsection{Optimization landscape}
\label{sec:optimization_landscape}

We do an experiment that tests the optimization landscape from the \textsc{Power Method} step in SPM. 
Recalling that we proved  \textsc{Power Method}  converges to a second-order critical point of \eqref{eq:optim}, we study how often the iterates converge to global maxima as compared to spurious second-order critical points (i.e., ones that are not global maxima). When SPM converges to a second order critical point $\x_* \in \mathbb{S}^{\dimL-1}$, we know it is a global maximum if and only if $F_{\cA}(\x_*)=1$.  By \cref{prop:subspacerank1generic}, we know up to which rank it holds that the global maxima of \eqref{eq:optim} are generically precisely the desired tensor components (up to sign).

Our experiment is as follows.   We consider fourth-order tensors ($m=4$ and $n=2$), fix the length $\dimL=20$ and vary $r$ from $120$ to $200$. Note that the rank threshold given by \cref{prop:subspacerank1generic} is $r= \binom{21}{2}=190$. For each rank, we sample $r$ independent and identically distributed standard Gaussian vectors $\{\a_{i}\}_{i=1}^r$ and form the corresponding subspace $\cA = \Span\{\a_{i}^{\otimes 2} : i=1,\dots,r\}$.  Since the experiment just tests \textsc{Power Method}, there is no need to generate coefficients $\lambda_i$ because the \textsc{Power Method} just operates on $\cA$.  We sample a single initialization vector $\x_0$ uniformly on $\mathbb{S}^{\dimL-1}$ and run \textsc{Power Method} starting at $\x_0$.
We regenerate both $\cA$ and $\x_0$ ten thousand times, and report the relative frequency of times the \textsc{Power Method}:

\vspace{0.3em}

\begin{enumerate}
	\item[{\fcolorbox{black}[rgb]{0, 0.3470, 0.6410}{\phantom{c}}}]
	converged to a tensor component: $\|\x_* - \y\|\le 10^{-10}$ for some $\y\in \{\pm\a_{i}\}_{i=1}^r$;\smallskip%
	\item[{\fcolorbox{black}[rgb]{0.9790, 0.8040, 0.2250}{\phantom{c}}}] 
	converged to a second-order critical point $\x_*$ that is not a global maximum: \smallskip\\
    \phantom{ }\hspace{.6cm} $\displaystyle F_{\cA}(\x_*) < 1 - 10^{-10}$;\smallskip 
	\item[{\fcolorbox{black}[rgb]{0.7500, 0.2750, 0.0980}{\phantom{c}}}]
	converged to a global maximum $\x_*$ that is not a tensor component: \smallskip\\
    \phantom{ }\hspace{.6cm} $\displaystyle F_{\cA}(\x_*) \ge 1 - 10^{-10} \qtext{and}\|\x_* - \y\|> 10^{-10} \text{ for all } \y\in \{\pm\a_{i}\}_{i=1}^r;$ \smallskip
	\item[{\fcolorbox{black}[rgb]{0.4660, 0.6740, 0.1880}{\phantom{c}}}]
	did not converge in $5000$ power method iterations: $\|\x_{5000}-\x_{4999}\|>10^{-10}$.
	\bigskip
\end{enumerate}

\begin{figure}
	\centering
	\includegraphics[width=.9\textwidth]{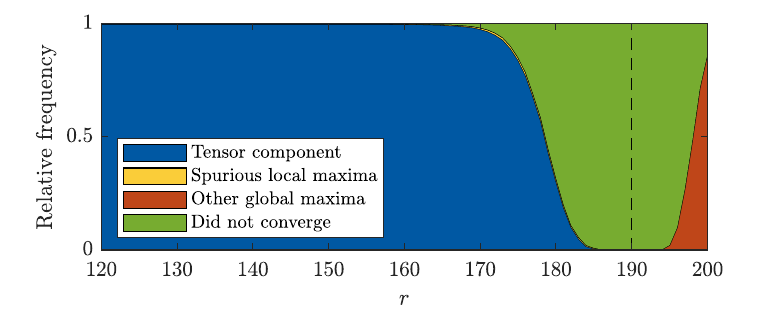}
	\caption{Relative frequency of convergence of the \textsc{Power Method} at different ranks $r$ over $10000$ trials when $n=2$ and $d=20$.  Each run, the CP components are random and we use only one random initialization (with no re-initializations). The dashed line shows the threshold given by \cref{prop:subspacerank1generic}.}
	\label{fig:pmlocalmaxfreq}
\end{figure}

The results are shown in \cref{fig:pmlocalmaxfreq}.  For all values of $r$ smaller than $140$, \textsc{Power Method} converges to one of the tensor's components every time over the $10000$ trials. 
Between ranks $140$ to $160$ the frequency is at least $99.8\%$, and it decreases to $97.6\%$ at $r=170$. Then, we observe a sharp transition when the rank varies between $170$ and $190$. Near the cutoff, many runs do not converge.  Although by \cref{thm:power} the runs would eventually  converge, \textsc{Power Method} needs more than $5000$ iterations.
We note that the width of the transition in \cref{fig:pmlocalmaxfreq} is $\approx \dimL$. The results suggest that if the rank scales like $c_n \dimL^n$, for a constant $c_n<\frac{1}{n!}$ (since $\binom{\dimL+n-1}{n} \approx \frac1{n!}\dimL^n$), \textsc{Power Method} converges to a CP component from a \textit{single} initialization with high probability.  Therefore, the optimization landscape for \eqref{eq:optim} seems well-behaved.

In \cite{kileel2021landscape} we investigate this phenomenon and characterize local maxima of the SPM functional.
There we prove that if $\{\a_{i}\}_{i=1}^r$ are drawn i.i.d. uniformly from $\mathbb{S}^{d-1}$, then with high probability should spurious local maxima exist their function value is rather small: $\mathcal{O}(r \log^n(r)/d^n)$.
We also provide theoretical guarantees under deterministic frame conditions, holding for example when $\a_i$ are approximately orthogonal.

\subsection{Real world application}
\label{sec:real_world}

In this subsection, we test SPM on a real world application. 
Our purpose is to show effectiveness of SPM in a setting where we do not control the data generation.

Specifically, we use SPM for eye blink artifact removal in real electroencephalogram (EEG) datasets \cite{li2006blinking}. A typical model considered for this purpose is Independent Component Analysis (ICA) \cite{hyvarinen2004independent}. This is a statistical model that assumes the data arises from a linear combination of independent data sources. Formally, ICA assumes that data samples $\x_{1},\dots \x_{p} \in \R^d$ arise from the following statistical model:
$$\x_{i} = \Mx{A} \Vc{s}_{i},\quad i=1,\dots,p,$$

\begin{figure}
	\centering
	\begin{subfigure}[t]{0.32\textwidth}
		\centering
		\includegraphics[width=\textwidth]{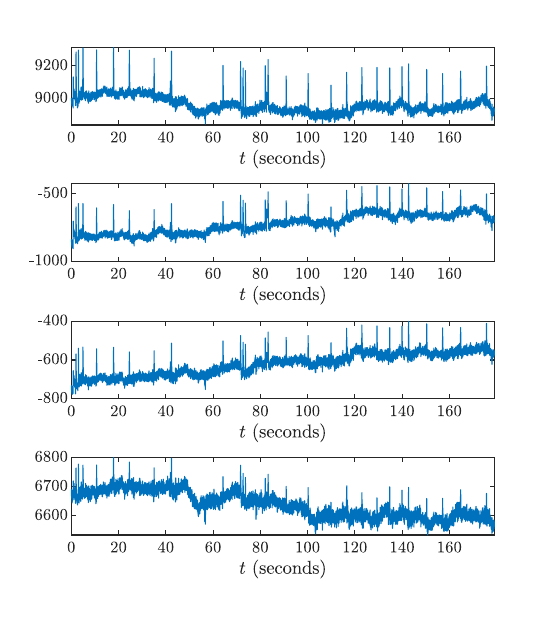}
		\caption{4 out of the 64 EEG signals.}
		\label{fig:EEG_data}
	\end{subfigure}
	\hfill
	\begin{subfigure}[t]{0.32\textwidth}
		\centering
		\includegraphics[width=\textwidth]{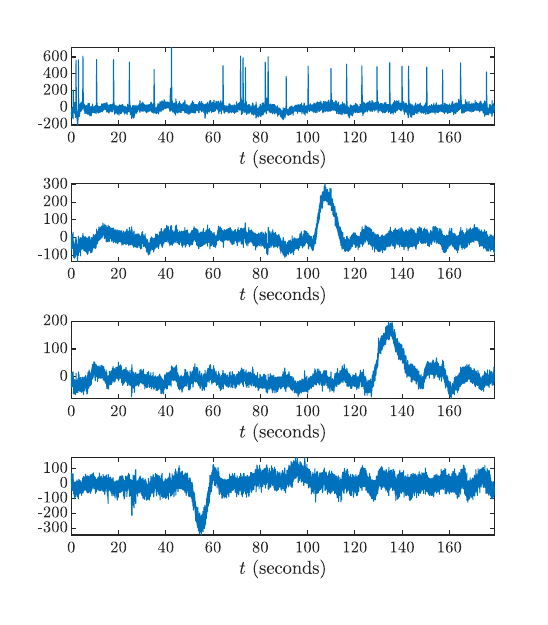}
		\caption{4 out of 64 ICA Components learnt using SPM to decompose the  cumulant tensor.}
		\label{fig:ICA_components}
	\end{subfigure}
	\hfill
	\begin{subfigure}[t]{0.32\textwidth}
		\centering
		\includegraphics[width=\textwidth]{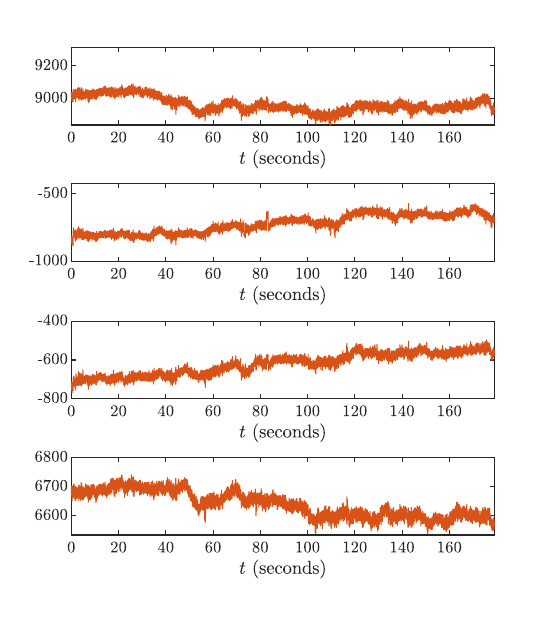}
		\caption{EEG signals of \cref{fig:EEG_data}, denoised by removing the learnt eye-blink artifact component.}
		\label{fig:EEG_data_denoised}
	\end{subfigure}
	\caption{Real EEG data which is analyzed and denoised assuming an ICA model. The EEG data in \cref{fig:EEG_data} contains eye-blink artifacts, which are removed in \cref{fig:EEG_data_denoised} by finding the eye-blink artifact signal using SPM (\cref{fig:ICA_components}).}
\label{fig:ICA}
\end{figure} 
\noindent where $\Mx{A}$ is a $d\times r$ matrix, called the coefficient or mixing matrix, and $\Vc{s}_{1},\dots,\Vc{s}_{p}\in \R^r$ are i.i.d. random vectors with independent entries. Although this is not a parametric statistical model, the independence condition between the elements of $\Vc{s}$ allows us to recover the mixing matrix. Moreover, if $r\le d$, the original independent components may be obtained from the data by multiplication with the pseudo-inverse of $\Mx{A}$. More specific to our setting, the higher-order cumulants of the ICA model are symmetric tensors with a low-rank CP structure \cite{foobi2007}, that allow for recovering $\Mx{A}$:
\begin{equation}\label{eq:ICA_cumulant}
\Tn{K}^{(n)} = \sum_{j=1}^r \kappa_j^{(n)} \Vc{a}_j^{\otimes n}.
\end{equation}
Here, $\Tn{K}^{(n)}$ denotes the model cumulant, $\Vc{a}_1,\dots, \Vc{a}_r$ are the columns of $\Mx{A}$, and $\kappa_j^{(n)}$ is the $n$-th cumulant of the $j$-th element of $\Vc{s}$.

In this experiment, we consider the EEG dataset presented and first analysed in \cite{chavarriaga2010learning}, which is the 22\textsuperscript{nd} dataset available to download in \url{https://bnci-horizon-2020.eu/database/data-sets}. Our analysis is summarized by \cref{fig:ICA}. The dataset consists of recorded EEG signals of 6 subjects, over 2 sessions and 10 runs per session, of 3 minutes each. In each run, we observe the simultaneous EEG signal over 64 electrodes placed on the subject, recorded at a 512 Hz frequency. For this demonstration, we analyze only the first run of the first session of subject $1$. EEG signals recorded at four electrodes (out of 64) are plotted in \cref{fig:EEG_data}.

We form the empirical 3rd order cumulant, which is a $64 \times 64 \times 64$ symmetric tensor, and then decompose it using SPM, taking the low-rank decomposition \eqref{eq:ICA_cumulant} with $r=64$ as an  ansatz. With the approximation for the ICA mixing matrix $\Mx[\hat]{A} \in \R^{64 \times 64}$ obtained by SPM, we multiply the original real data by its approximate inverse\footnote{We consider an approximate inverse since the mixing matrix obtained by SPM is ill-conditioned.}, ($(\Mx[\hat]{A}'\Mx[\hat]{A}+0.01\Mx{I})^{-1} \Mx[\hat]{A}'$), to obtain the ICA components. Four of the $64$ ICA components obtained are plotted in \cref{fig:ICA_components}.

One observes that one of the ICA components (the first plot in \cref{fig:ICA_components}) contains the eye-blink artifacts of the signals, which are caused by the subject blinking during the recording. Because we learned the mixing coefficients in $\Mx[\hat]{A}$, we may subtract this ICA component, multiplied by its corresponding mixing coefficients, to effectively remove the eye-blink artifacts from the data. We remove the artifacts in this manner for the signals plotted in \cref{fig:EEG_data}, and plot the denoised signals in \cref{fig:EEG_data_denoised}.

Comparing \cref{fig:EEG_data,fig:EEG_data_denoised}, we conclude that SPM succeeded in removing the eye-blink artifacts and obtained relevant ICA components from the data.

% --- End Inserted File ---

% ---- Inserted File ----
\section{Discussion}\label{sec:discussion}

This paper introduced a new algorithm, called SPM, for low-rank symmetric tensor CP decomposition. 
The new algorithm is performant: it is faster than some state-of-the-art algorithms, at least in experiments with tensors of moderate ranks.
Another advantage is that SPM does not require prior knowledge of the target rank.
The algorithm brings together ideas from algebraic geometry and nonconvex optimization.  As such, we were able to establish a rich mathematical foundation for its properties. 

Several aspects of SPM warrant further analysis. We study the optimization landscape of SPM \eqref{eq:optim} in a follow-up paper \cite{kileel2021landscape}, where we obtain nonconvex optimization guarantees for decomposing low-rank and approximately low-rank tensors by SPM, in a random overcomplete setting as well as under deterministic frame conditions.  
Algorithmically, there are many possible  extensions of SPM worth exploring.  These include modifying it to decompose non-symmetric tensors or calculate block term decompositions.  We also want to extend SPM to settings where the tensor is presented in a specialized manner, for example as a moment tensor of given data samples \cite{sherman2019estimating}. 
%
% --- End Inserted File ---

\section*{Acknowledgements}
	We thank Emmanuel Abbe, Amir Ali Ahmadi, Tamir Bendory, Nicolas Boumal, Tamara G. Kolda, and Amit Singer for helpful comments. 
	J.K. was supported in part by the Simons Collaboration on Algorithms and Geometry, NSF DMS 2309782, NSF DMS 2436499, NSF CISE-IIS 2312746 and DE SC0025312. J.P. was supported in part by NSF BIGDATA IIS-1837992 and a start-up grant from the University of Georgia. J.K. and J.P. were both supported in part by start-up grants to J.K. from the College of Natural Sciences and Oden Institute at UT Austin.

\bibliographystyle{siamplain}

%%\\n

% ---- Inserted File ----
\appendix

\section{Additional Proofs}

\subsection{Proof of \cref{lem:shift}}
\label{app:proofshift}
\begin{proof}
The Euclidean Hessian of $F_{\cA}$  is  given by
	\begin{equation*}
\nabla^2 F_{\cA}(\x)=2n\sum_{i=1}^r n(\U_i \cdot \x^{\otimes (n-1)})(\U_i \cdot \x^{\otimes (n-1)})^{\Tr} + (n-1) \dotpb{\U_i}{\x^{\otimes n}}\U_i \cdot \x^{\otimes (n-2)},
\end{equation*}
where  $\U_1,\dots,\U_r \in \cS_\dimL^n$ are an orthonormal basis of $\cA$.
For the rest of the proof denote $\x\y:=\x\otimes \y$ and $\x^n := \x^{\otimes n}$.
Let $\x,\y\in\R^\dimL$ such that $\|\x\|=\|\y\|=1$. We have
\begin{align}
\nonumber \frac1{2n} \y^{\Tr} \nabla^2 F_{\cA}(\x) \y
&= n\sum_{i=1}^r \dotpb{\U_i}{\x^{n-1} \y}^2 + (n-1) \sum_{i=1}^r \dotpb{\U_i}{ \x^{n}}\dotpb{\U_i}{\x^{n-2} \y^{2}}\\
\label{eq:Hexpansion} &=n \|P_\cA(\x^{n-1} \y)\|^2 + (n-1) \dotpb{P_\cA(\x^{n})}{\x^{n-2} \y^{2}},\\
\label{eq:Hexpansion_ineq} &\ge (n-1) \dotpb{P_\cA(\x^{n})}{\x^{n-2} \y^{2}},
\end{align}
where \eqref{eq:Hexpansion} follows from \eqref{eq:PAdef2}. 
Define $\y[\bar] \in \R^\dimL$ and $\alpha,\beta\in \R$ such that $\langle \bar \y,  \x \rangle = 0$, $\|\y[\bar]\|=1$, $\alpha^2 + \beta^2=1$ and $\y = \alpha \y[\bar] + \beta \x$. 
Substituting this expression for $\y$ into \eqref{eq:Hexpansion_ineq}, and rearranging using $P_{\mathcal{A}}$ is self-adjoint and $P_{\mathcal{A}}^2 = P_{\mathcal{A}}$, %
\begin{equation*}\label{eq:c3}
\frac1{2n} \y^{\Tr} \nabla^2 F_{\cA}(\x) \y \ge (n-1)\dotpb{P_\cA(\x^{n})}{\beta^2\x^{n} + 2\alpha \beta \x^{n-1}\y[\bar] +\alpha^2 \x^{n-2}\y[\bar]^2}.
\end{equation*}
\def\newnu{z}
\!\!\!\! Define $\newnu=\|P_\cA(\x^{n})\|^2=\dotpb{P_\cA(\x^{n})}{\x^{n}}\ge \nu$.  Letting $Y = 2\alpha \beta \x^{n-1} \y[\bar] +\alpha^2 \x^{n-2}\y[\bar]^2$, we have
\begin{align}
\frac1{2n} \y^{\Tr} \nabla^2 F_{\cA}(\x) \y &\ge (n-1)\beta^2 \newnu +  (n-1)\dotpb{P_\cA(\x^{n})}{Y} \nonumber \\
&= (n-1)\beta^2 \newnu + (n-1)\dotpb{P_\cA(\x^{n})}{\Sym(Y)}, \label{eq:good-bound-hessian}
\end{align}
where the last equation follows from $P_\cA(\x^{n})$ being a symmetric tensor and \cref{lemma:sym_orthogonal_projection}. Again using \cref{lemma:sym_orthogonal_projection}, \cref{lemma:inner_tensor_powers} and $ \langle \y[\bar],  \x \rangle = 0$, 
we obtain $\dotpb{\x^{n}}{\Sym(Y)}=\dotpb{\x^{n}}{Y}= 0$. Thus the following Bessel's inequality holds:
\begin{align*}
\|P_\cA(\x^n)\|^2 &\ge \left\langle P_\cA(\x^n), \x^n \right\rangle^2 + \Big{\langle} P_\cA(\x^n), \frac{\Sym(Y)}{\|\Sym(Y)\|} \Big{\rangle}^2\!,
\end{align*}
which implies
$$\left\langle P_\cA(\x^n), \Sym(Y) \right\rangle \ge - \sqrt{\newnu-\newnu^2}\|\Sym(Y)\|.$$
We plug this in \eqref{eq:good-bound-hessian} to obtain
\begin{align}
\label{eq:bound-with-nu}\frac1{2n} \y^{\Tr} \nabla^2f(\x) \y &\ge (n-1)\beta^2 \newnu - (n-1)\sqrt{\newnu-\newnu^2} \|\Sym(Y)\|.
\end{align}
Now we calculate
\begin{align*}
\|\Sym(Y)\|^2 &= \left\|2\alpha \beta \Sym(\x^{n-1} \y[\bar]) +\alpha^2 \Sym(\x^{n-2}\y[\bar]^2)\right\|^2\\
&= 4\alpha^2 \beta^2 \|\Sym(\x^{n-1} \y[\bar])\|^2 +4\alpha^3 \beta \dotpb{\Sym(\x^{n-1} \y[\bar])}{\Sym(\x^{n-2} \y[\bar]^2)} \\
&\hspace{200pt}+ \alpha^4 \| \Sym(\x^{n-2}\y[\bar]^2)\|^2.
\end{align*}
We start by calculating $\|\Sym(\x^{n-2} \y[\bar]^2)\|^2$. 
Note that 
\begin{equation*}
\Sym(\x^{n-2} \y[\bar]^2) =  \binom{n}{2}^{\! -1} \left(\x^{n-2} \y[\bar]^2 + \ldots + \y[\bar]^2 \x^{n-2}\right),
\end{equation*}
where the sum includes all $\binom{n}{2}$ terms with $\y[\bar]$ in $2$ of the $n$ positions.  The terms are pairwise orthogonal and  of unit norm, because $\langle \y[\bar], \x \rangle = 0$ and $\|\x\|=\|\y[\bar]\|=1$.  So
\begin{equation*}
\|\Sym(\x^{n-2} \y[\bar]^2)\|^2=\binom{n }{ 2}^{\!\! -1}=\frac{2}{n(n-1)}.
\end{equation*}
Analogously, $\|\Sym(\x^{n-1} \y[\bar])\|^2=\frac1n$ and $\dotpb{\Sym(\x^{n-2} \y[\bar]^2)}{\Sym(\x^{n-1} \y[\bar])}=0$.  Therefore
\begin{equation}
\label{eq:symy_norm}
\|\Sym(Y)\|^2=\frac{4\alpha^2\beta^2}{n}+\frac{2\alpha^4}{n(n-1)}.
\end{equation}
Now plugging in \eqref{eq:symy_norm} in \eqref{eq:bound-with-nu}, 
letting $t = \beta^2$ and noting that $\alpha^2 = 1-t$, we have
	\begin{align*}
	\nonumber\frac1{2n} \y^{\Tr} \nabla^2f(\x) \y &\ge (n-1)t \newnu - (n-1)\sqrt{\newnu-\newnu^2}\sqrt{\frac{4t(1-t)}{n} + \frac{2(1-t)^2}{n(n-1)}} \\ 
	\nonumber&=: g(t,\newnu) \ge \min_{t\in [0,1]} g(t,\newnu)
	\end{align*}
	We now estimate this minimum.  We overview the computation, but omit the full details for brevity. First we check, by calculating the second derivative that, for $\newnu$ fixed, $g$ is a convex function of $t$. Then we find the minimum in $t$ by solving $\tfrac{\partial g}{\partial t}=0$, which amounts to solving a quadratic equation in $t$. Finally, because this is constrained minimization, we have to check if the minimum lies inside the interval, else the minimum is achieved at the boundary. This leads to a case by case expression:
	$$\tilde g(\newnu):=\min_{t\in [0,1]} g(t,\newnu) = \begin{cases}
	-\sqrt{\frac{2(n-1)\newnu(1-\newnu)}{n}}&\text{if }\newnu \ge \frac{2 (n-2)^2}{3 n^2-9 n+8}\\
	-\frac{(n-1) \left((n-2) \newnu-\sqrt{\frac{(n-1) \newnu \left(4n-6+\newnu(n^2 -5n +6)\right)}{n}}\right)}{2 n-3}
	&\text{if }\newnu < \frac{2 (n-2)^2}{3 n^2-9 n+8}.

\	\end{cases}$$
	It can be checked by taking derivatives that $\tilde g(z)$ is also a convex function, and that we always have $\frac{2 (n-2)^2}{3 n^2-9 n+8} \le \frac23$. Furthermore, $\newnu\ge \nu$, hence
	\begin{align}
	\nonumber\frac1{2n} \y^{\Tr} \nabla^2 F_{\cA}(\x) \y &\ge \tilde g(\newnu)\ge \min_{\newnu\in [\nu, 1]} \tilde g(\newnu) \nonumber \\
	&\ge\begin{cases}
	\tilde g(\nu)&\text{if }\nu > \frac23\\
	\tilde g\left(\frac23\right) + (\nu-\frac23)\tilde g'\left(\frac23\right)&\text{if }\nu \le \frac23
	\end{cases} \nonumber \\
	&= -\sqrt{\frac{n-1}{n}}h(\nu) \label{eq:hard-hess},
\	\end{align}
where $h(\nu)$ is defined in \eqref{eq:gamma_k}.  Therefore \eqref{eq:hessian_psd} holds.

For the last sentence of \cref{lem:shift}, by homogeneity it suffices to check that for  $\x \in \mathbb{S}^{d-1}$  the eigenvalues of $\nabla^2(F_{\cA}(\x) + \gamma(\x^{\Tr} \x)^n)$ are strictly positive.
We compute 
\begin{equation} \label{eq:my-easy-hess}
\nabla^2{(}\gamma (\x^{\Tr} \x)^n{)} = 4n(n-1) \gamma (\x^{\Tr} \x)^{n-2} \x \x^{\Tr} + 2n \gamma \|\x\|^{2n-2} \Id =  4n(n-1) \gamma \x \x^{\Tr} + 2n \gamma \Id,
\end{equation}
From \eqref{eq:my-easy-hess} and \eqref{eq:hard-hess} we obtain
$$
\frac{1}{2n} \y^{\Tr} \nabla^2(F_{\cA}(\x) + \gamma (\x^{\Tr} \x)^n) \y \geq -\sqrt{\frac{n-1}{n}} h(\nu) + \gamma \geq -\sqrt{\frac{n-1}{n}} + \gamma,
$$
which is strictly positive if $\gamma > \sqrt{\frac{n-1}{n}}$.  This finishes the proof of \cref{lem:shift}.
 \end{proof}

\subsection{Proof of \cref{lem:calc}}
\label{app:proofcalc}
\begin{proof}
Let $\x \in \mathbb{S}^{d-1}$. Simple calculations show that
\begin{small}
\begin{align*}
    \frac{\partial \Psi_i}{\partial x_j}  =  \frac{\partial}{\partial x_j} \left( {{\frac{\partial G(\x)}{\partial x_i}}} {/} {\| \nabla G(\x) \|} \right) =  \left( \frac{\partial^2 G(\x)}{\partial x_i \partial x_j} \| \nabla G(\x) \| - \frac{\partial G(\x)}{\partial x_i} \frac{\partial \| \nabla G(\x) \| }{\partial x_j} \right) \! {/} \| \nabla G(\x) \|^2, \!\!\!
\end{align*}
\end{small}
and 
\begin{align*}
\frac{\partial \| \nabla G(\x) \| }{\partial x_j} \, = \, \| \nabla G(\x)\|^{-1} \sum_{k=1}^d \frac{\partial G(\x)}{\partial x_k} \frac{\partial^2 G(\x)}{\partial x_j \partial x_k}  \, = \, \frac{{(}\nabla^2 G(\x) \nabla G(\x){)}_j}{\| \nabla G(\x) \|}.  
\end{align*}
Therefore, 
\begin{align*}
\frac{\partial \Psi_i}{\partial x_j}  =  \frac{(\nabla^2 G(\x))_{ij}}{\| \nabla G(\x)\|} - \frac{ (\nabla G(\x) \nabla G(\x)^{\Tr} \nabla^{2} G(\x))_{ij}}{\| \nabla G(\x) \|^3} = \left((\Id - \Psi(\x)\Psi(\x)^{\Tr}) \frac{\nabla^2 G(\x)}{\| \nabla G(\x) \|} \right)_{ij}.
\end{align*}
Viewing $D \Psi$ as a linear map between tangent spaces to the sphere, its domain is the tangent space to $\mathbb{S}^{d-1}$ at $\x$. 
By a harmless abuse we express this  by right-multiplying with the projector onto this tangent space.  It yields 
\begin{equation*}
D\Psi(\x) = (\Id - \Psi(\x) \Psi(\x)^{\Tr}) \frac{\nabla^2 G(\x)}{\| \nabla G(\x)\|} (\Id - \x \x^{\Tr}),
\end{equation*}
as wanted.

Next let $\x_* \in \mathbb{S}^{d-1}$ be a first-order critical point of \eqref{eq:maxf}.  Then $\Psi(\x_*) = \x_*$, thus
\begin{equation}\label{eq:my-app-eqn1}
D\Psi(\x_*) = (\Id - \x_* \x_*^{\Tr}) \frac{\nabla^2 G(\x_*)}{\| \nabla G(\x_*)\|} (\Id - \x_* \x_*^{\Tr}).
\end{equation}
Also $\x_*$ is a first-order critical point of \eqref{eq:maxf2}.
So the Riemmanian gradient satisfies $\operatorname{grad} G(\x_*) = 0$, implying (see \eqref{eq:def-riem-grad})  $\nabla G(\x_*) = (\x_*^{\Tr} \nabla G(\x_*)) \x_* = 2n G(\x_*) \x_* = 2n(F(\x_*) + \gamma)\x_*$, whence 
$\| \nabla G(\x_*) \| = 2n(F(\x_*) + \gamma)$.
We also have $\nabla^2 G(\x_*) = \nabla^2 F(\x_*) + \nabla^2{(}\gamma (\x_*^{\Tr} \x_*)^n{)} = \nabla^2 F(\x_*) + 4n(n-1) \gamma \x_* \x_*^{\Tr} + 2n \gamma \Id$ by \eqref{eq:my-easy-hess}.
Substituting into \eqref{eq:my-app-eqn1} we obtain
\begin{equation} \label{eq:so-close-3}
D\Psi(\x_*) = (\Id - \x_* \x_*^{\Tr}) \frac{\nabla^2 F(\x_*)}{2n(F(\x_*)+\gamma)}(\Id - \x_* \x_*^{\Tr}) + \frac{\gamma}{F(\x_*) + \gamma} (\Id - \x_* \x_*^{\Tr}).
\end{equation}
On the other hand, as in \eqref{eq:def-riem-hess} the Riemmanian Hessian is 
\begin{align} \label{eq:so-close-4}
\operatorname{Hess} F(\x_*) &= (\Id - \x_* \x_*^{\Tr}) \nabla^2 F(\x_*) (\Id - \x_* \x_*^{\Tr}) - (\x_*^{\Tr} \nabla F(\x_*)) (\Id - \x_* \x_*^{\Tr}) \nonumber \\
&= (\Id - \x_* \x_*^{\Tr}) \nabla^2 F(\x_*) (\Id - \x_* \x_*^{\Tr}) - 2n F(\x_*) (\Id - \x_* \x_*^{\Tr}).
\end{align}
Comparing \eqref{eq:so-close-3} and \eqref{eq:so-close-4} we conclude \eqref{eq:jac-firstx}.  This proves \cref{lem:calc}.  
 \end{proof}

\section{Background Statements}

\subsection[Schneider-Uschmajew's convergence result]{Convergence result \cite[Thm.~2.3]{schneider2015convergence}} \label{app:convergence-result}
For simplicity we state a special case of Schneider-Uschmajew's result, which is enough for our purposes.  

\begin{theorem} \label{thm:meta-converge}   \textup{(\cite[Thm.~2.3]{schneider2015convergence})}
Let $\mathcal{M} \subset \R^d$ be a compact Riemannian submanifold that is locally the image of a real analytic map out of a Euclidean space. 
Let $\mathcal{D} \subset \R^d$ be an open neighborhood of $\mathcal{M}$.  Suppose $F : \mathcal{D} \rightarrow \R$ is a real analytic function that is bounded below.  Let $\operatorname{grad} F$ denote the Riemannian gradient of the restriction of $F$ to $\mathcal{M}$. 
 Consider the  problem
\begin{equation} \label{eq:f-min-prob}
\max_{\x \in \mathcal{M}} F(\x).
\end{equation}
Suppose that $(\x_k)_{k=1}^{\infty} \subset \mathcal{M}$ is a sequence satisfying the following  three assumptions:
\begin{enumerate}[label=\textup{(\textbf{A\arabic*})},ref={\textup{(A\arabic*)}}]
\item \label{enum:SU_A1}
There exists $\sigma > 0$ such that for $k$ large enough, %
\begin{displaymath}
    F(\x_{k+1}) - F(\x_k) \geq \sigma \| \operatorname{grad} F(\x_k) \| \|\x_{k+1} - \x_k \|.
\end{displaymath}
\item \label{enum:SU_A2}
For large enough $k$, %
\begin{displaymath}
    \operatorname{grad} F(\x_k) = 0 \implies \x_{k+1} = \x_k.
\end{displaymath}
\item \label{enum:SU_A3}
There exists $\rho > 0$ such that for large enough $k$, %
\begin{displaymath}
    \|\x_{k+1} - \x_k \| \geq \rho \| \operatorname{grad} F(\x_k) \|.
\end{displaymath}
\end{enumerate}
Then $(\x_k)_{k=1}^{\infty}$ converges to a point $\x_* \in \mathcal{M}$ and there exist constants $C > 0$ and $\tau >1$ such that 
\begin{equation}
   \|\x_k - \x_* \| \leq C k^{-\tau}
\end{equation}
for all $k$.  Moreover, $\|\operatorname{grad} F(\x_k)\| \rightarrow 0$ as $k \rightarrow \infty$.
\end{theorem}

\subsection{Center-stable manifold theorem}\label{app:statementcsm}
This is stated for open sets in Euclidean space in \cite{shub2013global}.  However, by taking charts it holds on manifolds as below.

\begin{theorem}\textup{(\cite[Thm.~III.7(2)]{shub2013global})}\label{thm:CMT}
Let $\mathcal{M}$ be a smooth manifold, $\Psi: \mathcal{M} \rightarrow \mathcal{M}$ a local diffeomorphism, and $\x \in \mathcal{M}$ a fixed point of $\Psi$.  Then there exist an open neighborhood $B_{\x} \subset \mathcal{M}$ of $\x$ and a smoothly embedded disk $D_{\x} \subset \mathcal{M}$ containing $\x$ such that the following properties hold:
\begin{itemize}
\item $\{ \mathbf{y} \in \mathcal{M} : \Psi^k(\mathbf{y}) \in B_{\x} \,\, \forall k \geq 0 \} \subset D_{\x}$;
    \item $\dim(D_{\x})$ is the number of linearly independent eigenvectors of $D\Psi(\x)$ whose eigenvalues don't exceed $1$ in magnitude.
\end{itemize}
\end{theorem}

% --- End Inserted File ---
\end{document}